\documentclass[11pt,letterpaper]{amsart}
\usepackage{graphicx}
\usepackage{hyperref}
\usepackage{tikz}
\usepackage{xcolor}
\usepackage{comment}
\usepackage{mathrsfs}
\usepackage[T1]{fontenc}
\usepackage{xypic}
\usepackage{overpic}

\usepackage{amsmath,amsthm,amssymb,amscd,enumerate}

\usetikzlibrary{shapes.geometric, arrows, positioning}

\newtheorem{thm}{Theorem}[section]
\newtheorem{prop}[thm]{Proposition}
\newtheorem{lem}[thm]{Lemma}
\newtheorem{cor}[thm]{Corollary}

\newtheorem{mthm}{Theorem}

\theoremstyle{definition}
\newtheorem{definition}[thm]{Definition}

\theoremstyle{remark}

\newtheorem{remark}[thm]{Remark}

\numberwithin{equation}{section}

\allowdisplaybreaks[1]

\newcommand{\RR}{\mathbb{R}}
\newcommand{\ZZ}{\mathbb{Z}}

\newcommand{\CC}{\mathbb{C}}

\newcommand{\qd}{\mathscr{A}}
\newcommand{\ee}{\epsilon}
\newcommand{\bb}{\mathfrak{b}}

\newcommand{\tg}{\widetilde{h}}

\newcommand{\bG}{G}

\newcommand{\id}{\mathrm{id}}
\newcommand{\GG}{\Gamma}
\newcommand{\tR}{\widetilde{R}_0}
\newcommand{\tRb}{\widetilde{R}_{b,0}}

\newcommand{\Symp}{\mathrm{Symp}^c}
\newcommand{\Ham}{\mathrm{Ham}^c}
\newcommand{\Diff}{\mathrm{Diff}^c}
\newcommand{\tHam}{\widetilde{\mathrm{Ham}^c}}
\newcommand{\tSymp}{\widetilde{\mathrm{Symp}_0^c}}
\newcommand{\tDiff}{\widetilde{\mathrm{Diff}_0^c}}

\newcommand{\Supp}{\mathrm{Supp}}
\newcommand{\flux}{\mathrm{Flux}}
\newcommand{\Flux}{\mathrm{Flux}}
\newcommand{\vol}{\mathrm{vol}}
\newcommand{\Hom}{\mathrm{Hom}}
\newcommand{\Vol}{\mathrm{Vol}}

\newcommand{\Cal}{\mathrm{Cal}}

\newcommand{\Sp}{\mathrm{Sp}}
\newcommand{\UU}{\mathrm{U}}
\newcommand{\CCC}{\mathrm{C}}
\newcommand{\QQQ}{\mathrm{Q}}
\newcommand{\HHH}{\mathrm{H}}

\newcommand{\iotap}{\iota_1}
\newcommand{\iotau}{\iota_2}

\newcommand{\cov}{p} 

\newcommand{\sqm}{\mathfrak{S}}
\newcommand{\bbeth}{\mathscr{F}}

\newcommand{\ioP}{\iota}
\newcommand{\iPN}{\iota}

\newcommand{\tL}{\widetilde{G}}
\renewcommand{\L}{G}
\newcommand{\tio}{\widetilde{i_0}}
\newcommand{\til}{\widetilde{i_1}}

\newcommand{\kwsk}{\color{olive}}

\newcommand{\Gg}{G}
\newcommand{\Ng}{N}

\newcommand{\Gamg}{\Gamma}

\begin{document}

\title[Shelukhin's quasimorphism and Reznikov's class]{Non-extendablity of Shelukhin's quasimorphism and non-triviality of Reznikov's class}

\author[M. Kawasaki]{Morimichi Kawasaki}
\address[Morimichi Kawasaki]{Department of Mathematics, Faculty of Science, Hokkaido University, North 10,
West 8, Kita-ku, Sapporo, Hokkaido 060-0810, Japan.}
\email{kawasaki@math.sci.hokudai.ac.jp}

\author[M. Kimura]{Mitsuaki Kimura}
\address[Mitsuaki Kimura]{Department of Mathematics, Osaka Dental University
 8-1 Kuzuha-hanazono-cho, Hirakata, Osaka 573-1121, Japan}
\email{kimura-m@cc.osaka-dent.ac.jp}

\author[S. Maruyama]{Shuhei Maruyama}
\address[Shuhei Maruyama]{School of Mathematics and Physics, College of Science and Engineering, Kanazawa University, Kakuma-machi, Kanazawa, Ishikawa, 920-1192, Japan}
\email{smaruyama@se.kanazawa-u.ac.jp}

\author[T. Matsushita]{Takahiro Matsushita}
\address[Takahiro Matsushita]{Department of Mathematical Sciences, Faculty of Science, Shinshu University, Matsumoto, Nagano, 390-8621, Japan}
\email{matsushita@shinshu-u.ac.jp}

\author[M. Mimura]{Masato Mimura}
\address[Masato Mimura]{Mathematical Institute, Tohoku University, 6-3, Aramaki Aza-Aoba, Aoba-ku, Sendai 9808578, Japan}
\email{m.masato.mimura.m@tohoku.ac.jp}

\makeatletter
\@namedef{subjclassname@2020}{%
\textup{2020} Mathematics Subject Classification}
\makeatother


\begin{abstract}
Shelukhin constructed a quasimorphism on the universal covering of the group of Hamiltonian diffeomorphisms for a general closed symplectic manifold.
In the present paper, we prove the non-extendability of that quasimorphism for certain symplectic manifolds, such as a blow-up of torus and the product of a surface of genus at least two and a closed symplectic manifold.
As its application, we prove the non-vanishing of Reznikov's characteristic class for the above symplectic manifolds.
\end{abstract}

\maketitle


\section{Introduction}\label{intr section}

\subsection{Backgrounds of our main results}
A real-valued function $\phi \colon G \to \RR$ on a group $G$ is called a {\it quasimorphism} if
\[
 D(\phi):= \sup_{g,h \in G}|\phi(gh) - \phi (g) - \phi(h)|<\infty.
\]
In symplectic geometry, quasimorphisms have many applications.
For a symplectic manifold $(M,\omega)$, we consider the following two natural transformation groups, the group $\Ham(M,\omega)$ of Hamiltonian diffeomorphisms and the identity component $\Symp_0(M,\omega)$ of the group of symplectomorphisms (see Subsection \ref{subsec=symp}).
We also consider their universal coverings $\tHam(M,\omega)$ and $\tSymp(M,\omega)$.
 In \cite{EP03}, Entov and Polterovich introduced the concept of Calabi quasimorphism, which is defined over the group of Hamiltonian diffeomorphisms (for example, see \cite[Subsection 1.1]{EP03}, \cite[Subsection 4.8]{PR}).
Entov and Polterovich constructed quasimorphisms on the group of Hamiltonian diffeomorphisms and obtained some applications on the Hofer metric and commutator length.
After their work, the notion of Calabi quasimorphism (and its generalizations) is widely applied to many problems (for example, non-displaceability of fiber of integrable systems \cite{EP09}, \cite{FOOO} and \cite{KO22} and Poisson bracket invariants \cite{BEP12} and \cite{S15}).
As excellent surveys of these topics, we refer the reader to \cite{E} and \cite{PR}.
One famous example of a Calabi quasimorphism is due to Py \cite{Py06}.
For a symplectic surface $(S,\omega)$ of genus at least two, Py constructed a Calabi quasimorphism $\mu_P\colon \Ham(S,\omega) \to \RR$, which we call \textit{Py's Calabi quasimorphism}.

In the present paper, we study the  ``extension problem'' of  ``invariant quasimorphism'' in symplectic geometry.
Let $\Gg$ be  a group and $\Ng$ a normal subgroup of $\Gg$.
A quasimorphism $\phi \colon \Ng \to \RR$ on $\Ng$ is said to be \emph{$\Gg$-invariant} if 
\[
\phi(g x g^{-1}) = \phi(x)
\]
holds for every $g \in \Gg$ and for every $x \in \Ng$.
For example, Py's Calabi quasimorphism $\mu_P$ is a $\Symp_0(S,\omega)$-invariant quasimorphism.

\begin{definition}\label{defn=Vspace}
Let $\Gg$ be a group and $\Ng$ a normal subgroup of $\Gg$. 
A $\Gg$-invariant quasimorphism $\mu$ is said to be \emph{extendable} to $\Gg$ if  there exists a quasimorphism $\hat\mu$ on $\Gg$ such that $\hat\mu|_\Ng=\mu$.
\end{definition}
The extension problem of quasimorphisms has been studied by several researchers.
Ishida \cite{IshidaThesis}, \cite{I} gave some sufficient condition for extendability of quasimorphisms and  Shtern \cite{Sh} gave an example of a non-extendable quasimorphism (as a survey, see \cite[Section 3]{KKMMMsurvey}).

Some of the authors considered the extension problem of Py's Calabi quasimorphism and obtained the following result.
\begin{thm}[\cite{KK}] \label{KKPy}
Let $(S,\omega)$ be a closed connected orientable surface whose genus is at least two with a symplectic form $\omega$.
Then,
Py's Calabi quasimorphism $\mu_P\colon \Ham(S,\omega) \to \RR$ is \emph{not} extendable to $\Symp_0(S,\omega)$.
\end{thm}

The work of \cite{KKMM} improved Theorem \ref{KKPy}.
More precisely, they proved the non-extendability of $\mu_P$ to a subgroup of $\Symp_0(S,\omega)$ smaller than $\Symp_0(S,\omega)$.
As its application, they proved the following.
\begin{thm}[{\cite[Theorem~1.1]{KKMM}}] \label{old_flux}
Let  $(S,\omega)$ be a closed connected orientable surface whose genus is at least two with a symplectic form $\omega$.
Assume that a pair $f,g \in  \Symp_0(S,\omega)$ satisfies $fg = gf$. Then
\[ \flux_\omega(f) \smile \flux_\omega(g) = 0\]
holds true. Here, $\flux_{\omega}\colon \Symp_0(S,\omega)\to \HHH^1(S;\RR)$ is the flux homomorphism \textup{(}see Subsection~\textup{\ref{subsec=symp}}\textup{)}, and $\smile \colon \HHH^1(S;\RR) \times \HHH^1(S;\RR) \to \HHH^2(S;\RR) \cong \RR$ denotes the cup product.
\end{thm}

\subsection{Main results}

Theorem~\ref{old_flux} implies that the extension problem of quasimorphisms may have extrinsic applications.
However, Py's Calabi quasimorphism is defined only for two-dimensional symplectic manifolds.
In the present paper, we consider the extension problem of Shelukhin's quasimorphism, which is defined for any finite volume symplectic manifold of arbitrary dimension.

For a closed symplectic manifold $(M,\omega)$, Shelukhin \cite[Section 3 (5)]{Shelukhin} constructed a quasimorphism $\sqm_M \colon \tHam(M,\omega) \to \RR$ (for precise construction, see Subsection \ref{subsec:sh_qm}).
The idea of his construction comes from Weinstein's action homomorphism and his quasimorphism is a generalization of another known quasimorphism constructed by Entov \cite{En04}.
To the best of the authors' knowledge, Shelukhin's construction of quasimorphisms seems to be the only method for constructing non-zero quasimorphisms on the universal covering of the group of Hamiltonian diffeomorphisms of a general closed symplectic manifold at present.

Our first main theorem is the following one.


\begin{mthm}[non-extendaility of Shelukhin's quasimorphism]\label{main_non_ext}
Let $(M,\omega)$ be one of the following:
\begin{enumerate}[$(1)$]
 \item The product of $(S,\omega_S)$ and $(N ,\omega_N)$, where $(S,\omega_S)$ is a closed connected orientable surface whose genus is at least two with a symplectic form $\omega_S$ and   $(N ,\omega_N)$ is a closed symplectic manifold.
\item The blow-up of the symplectic torus $(T^{2n},\omega)$ \textup{(}$n\geq2$\textup{)} as indicated in Theorem \textup{\ref{torus blow up}} \textup{(1)}.
\end{enumerate}
Then,
Shelukhin's quasimorphism $\sqm_M \colon \tHam(M,\omega) \to \RR$ is not extendable to $\tSymp(M,\omega)$.
\end{mthm}

One application of Theorem \ref{main_non_ext} is the non-triviality of a characteristic class constructed by Reznikov \cite{MR1694899}. This class has a close relation with Shelukhin's quasimorphism; see \cite[Section 3 (5)]{Shelukhin} and (\ref{bounded_coboundary}).

In what follows, we briefly present the definition of this characteristic class;  we will discuss it in more detail in Subsection~\ref{subsec:re_class}.
In \cite{MR1694899}, Reznikov defined some cocycles on the group of symplectomorphisms in the context of characteristic classes of foliated bundles, such as the Bott--Thurston cocycles representing the Godbillon--Vey classes.
One of the cocycles is defined as follows:

Let $(M,\omega)$ be a symplectic manifold and $\mathcal{J}(M,\omega)$ the space of $\omega$-compatible almost complex structures on $M$.
We note that $\mathcal{J}(M,\omega)$ can be regarded as an ``infinite-dimensional symplectic manifold,'' whose ``symplectic form'' is denoted by $\Omega$.
For $J_0, J_1, J_2 \in \mathcal{J}(M,\omega)$, we can define a ``geodesic triangle'' $\Delta(J_0, J_1, J_2)$ on $\mathcal{J}(M,\omega)$.

For $g,h \in \Symp(M,\omega)$, we set
\[
  b_J (g,h) = \int_{\Delta(J, gJ, ghJ)} \Omega.
\]
  The \textit{Reznikov class} is the group cohomology class $R = [b_J] \in \HHH^2(\Symp(M,\omega))$.
  Let $R_0 \in \HHH^2(\Symp_0(M,\omega))$ be the Reznikov class restricted to $\Symp_0(M,\omega)$.

As their symbols suggest, the Reznikov class does not depend on the choice of $J$; see Remark~\ref{rem:indep}. 
The cocycle $b_J$ is  shown to be a coboundary on $\tHam(M, \omega)$, that is, the pullback of the Reznikov class to $\tHam(M, \omega)$ is trivial.
Reznikov proved that his cohomology class is non-trivial for high dimensional symplectic torus \cite[Subsection 3.5]{MR1694899} and a certain connected component of the representation variety ${\rm Hom}(\pi_1(S),{\rm SO}(3))/{\rm SO(3)}$ with the standard symplectic structure, where $S$ is a closed orientable surface of genus at least two  \cite[Subsection 3.6]{MR1694899}.  
Our second main theorem supplies other examples of symplectic manifolds for which the Reznikov class is non-trivial.

\begin{mthm}[non-triviality of the Reznikov class]\label{reznikov_main}
  Let $(M,\omega)$ be as in Theorem~\textup{\ref{main_non_ext}}.
  Then, the restriction $R_0 \in \HHH^2(\Symp_0(M,\omega);\RR)$ of the Reznikov class $R$ is non-zero.
  In particular, the Reznikov class $R \in \HHH^2(\Symp(M,\omega);\RR)$ is non-zero.
  \end{mthm}

We note that Theorem \ref{reznikov_main} is obtained as a corollary to Theorem \ref{main_non_ext}.

As exemplified by Gromov's famous non-squeezing theorem \cite{Gr85}, the comparison between diffeomorphisms preserving the symplectic form and those preserving the volume form is a problem of broad interest.
Our third main theorem detects a difference between them in terms of \textit{relative bounded cohomology} $\HHH_{/b}^\ast$ studied in \cite{KKMMM21}. 
Here, for a group $G$, $\HHH_{/b}^\ast(G)$ is the homology of the cochain complex $(\CCC^*(G)/\CCC_b^*(G), \delta)$ appearing in the exact sequence
 \[0 \to \CCC_b^*(G) \to \CCC^*(G) \to \CCC^*(G)/\CCC_b^*(G) \to 0\] 
 \textup{(}see Subsection~\textup{\ref{subsec=bdd_cohom}}\textup{)}.
For a $2n$-dimensional symplectic manifold $(M,\omega)$, $\Omega = \omega^n$ is a volume form.
Then, we can consider the map $\iota \colon \tSymp(M,\omega) \to \tDiff(M,\Omega)$, where $\tDiff(M,\Omega)$ is the universal covering of the identity component of the group of diffeomorphisms of $M$ preserving $\Omega$. 

\begin{mthm}[comparison between symplectic preserving and volume-preserving]\label{main_not_inj}
Let $(M,\omega)$ be one of the following:
\begin{enumerate}[$(1)$]
 \item The product of $(S,\omega_S)$ and $(N ,\omega_N)$, where $(S,\omega_S)$ is a closed connected orientable surface whose genus is at least two with a symplectic form $\omega_S$ and   $(N ,\omega_N)$ is a compact K\"{a}hler manifold whose complex dimension is $n-1$, where $n\geq2$.
\item The blow-up of the symplectic torus $(T^{2n},\omega)$ \textup{(}$n\geq2$\textup{)} as indicated in Theorem \textup{\ref{torus blow up}}.
\end{enumerate}

Then,  the induced map $\iota^* \colon \HHH_{/b}^2(\tDiff(M,\Omega)) \to \HHH_{/b}^2(\tSymp(M,\omega))$ is not injective, where $\Omega=\omega^n$.
\end{mthm}

The following corollary immediately follows from Lemma~\ref{lem=extendable} and Theorem~\ref{main_non_ext}.
\begin{cor}\label{main_non_section}
Let $(M,\omega)$ be one of the following:
\begin{enumerate}[$(1)$]
 \item The product of $(S,\omega_S)$ and $(N,\omega_N)$, where $(S,\omega_S)$ is a closed connected orientable surface whose genus is at least two with a symplectic form $\omega_S$ and $(N,\omega_N)$ is a closed symplectic manifold.
 \item The blow-up of the symplectic torus $(T^{2n},\omega)$ $(n\geq2)$ as indicated in Theorem \textup{\ref{torus blow up}}.
\end{enumerate}
Then, the flux homomorphism $\widetilde{\flux}_\omega \colon \widetilde{\Symp_0}(M, \omega) \to \HHH_c^1(M; \RR)$ has no section homomorphism.
\end{cor}

When $M$ is a closed connected symplectic surface of genus at least two, Corollary \ref{main_non_section} was proved in \cite{KK}.

At the end of this subsection, we mention the forthcoming work \cite{KKMMMbilinear} of the authors. 
There, we obtain a refinement of a part of Theorem~\ref{main_non_ext} (see \cite[Theorem B and Theorem D]{KKMMMbilinear}). 
More precisely, let $(M,\omega)$ be either a finite product of closed connected symplectic surfaces of genus at least one, or a blow-up of a symplectic torus. 
Then, given a group $P$ with $\widetilde{\Ham}(M,\omega)\leqslant P\leqslant \widetilde{\Symp_0}(M,\omega)$, we provide a characterization of the extendability of Shelukhin's quasimorphism $\sqm_M$ to $P$ in terms of the image of $P$ in $\HHH^1_c(M;\RR)$ under the flux homomorphism. 
In addition, we have a refinement of Corollary~\ref{cor:surface_Reznikov} (see \cite[Corollary~9.12]{KKMMMbilinear}). 
To obtain these refinements, in \cite{KKMMMbilinear} we build a theory of constructing a bilinear form on $\HHH^1_c(M;\RR)$ from $\sqm_M$.


\subsection{Notation and conventions} \label{subsec:notation}
\begin{enumerate}[$(1)$]
\item
We always assume that symplectic manifolds are connected. 
\item
For two symplectic manifolds $(M_1,\omega_1)$ and $(M_2,\omega_2)$, we define their product $(M_1,\omega_1) \times (M_2,\omega_2)$ as the symplectic manifold $(M_1\times M_2,\mathrm{pr}_1^\ast\omega_1+\mathrm{pr}_2^\ast\omega_2)$, where $\mathrm{pr}_1\colon M_1\times M_2\to M_1$ and $\mathrm{pr}_2\colon M_1\times M_2\to M_2$ are the first and second projections, respectively.
We often write $\omega_{M_1\times M_2}$ for the symplectic form $\mathrm{pr}_1^\ast\omega_1+\mathrm{pr}_2^\ast\omega_2$.
\item
Throughout the present paper, we always consider the $\RR$-coefficient cohomology.
Thus we omit the coefficient $\RR$ in singular cohomology and group cohomology.
\item
Let $(M,\omega)$ be a symplectic manifold and $G$ denote $\Ham(M,\omega)$ or $\Symp_0(M,\omega)$ (see Subsection \ref{subsec=symp}).
Then, we define the universal covering $\tilde{G}$ of $G$ by
\[\tilde{G}=\{[\gamma] \mid \gamma\colon [0,1] \to G \text{ is a smooth isotopy with }\gamma(0)=\id_M\}.\]
Here, $[\gamma]$ is the smooth homotopy class of $\gamma$ in $G$ relative to fixed ends.
 \item Let $(P,\omega_P)$, $(N,\omega_N)$  and $(M,\omega)$ be symplectic manifolds and $\ioP\colon P\times N\to M$ be an open symplecitc embedding of  $(P\times N ,\omega_{P\times N})$ to $(M,\omega)$.
Assume that $N$ is a closed manifold. 
Then,  we define a homomorphism $\ioP_\ast^{P,M}\colon \tHam\left(P,\omega_{P}\right) \to \tHam(M,\omega)$ as follows.
 For $\tg = \left[\{h^t\}_{t\in[0,1]}\right] \in \tHam\left(P,\omega_{P}\right)$, we define $\ioP_\ast^{P,M}(\tg) \in \tHam(M,\omega)$ by
  \[\ioP_\ast^{P,M}(\tg) = \left[\{\iota_\ast\left(h^t\times \id_N\right)\}_{t\in[0,1]}\right].\]
  Here, $\ioP_\ast\colon \Ham\left(P\times N,\omega_{P\times N}\right) \to \Ham(M,\omega)$ is the map induced from the open symplectic embedding $\ioP\colon P\times N\to M$. \label{item:emb}

\item For a manifold $M$, let $\HHH^1_c(M ; \RR)$ denote the first cohomology group  with compact support.
We note that  $\HHH^1_c(M ; \RR)=\HHH^1(M ; \RR)$ if $M$ is closed.
\item In the present paper, we in principle use the symbol $N$ for a (symplectic)
manifold, except Subsections \ref{subsec=qm} and \ref{subsec=bdd_cohom}.
In these subsections, $N$ denotes a normal subgroup of a group $G$.
\end{enumerate}


\subsection{Organization of the present paper}
Section~\ref{prel sect} provides the preliminaries.
In Section~\ref{sec_blow_up}, we compute the average Hermitian scalar curvature for symplectic blow-ups of tori, which is a key ingredient for the proof of Theorem~\ref{main_non_ext} (2).
In Section~\ref{sec:proof_main}, we prove Theorem~\ref{main_non_ext}.
In Section~\ref{reznikov sect}, we prove Theorem~\ref{reznikov_main} and Theorem~\ref{main_not_inj}.

\section{Preliminaries}\label{prel sect}

\subsection{Symplectic geometry}\label{subsec=symp}
In this subsection, we review some concepts in symplectic geometry that we will need in the subsequent sections.
For a more comprehensive introduction to this subject, we refer the reader to \cite{Ban97}, \cite{MS} and \cite{P01}.

Let $(M,\omega)$ be a 
symplectic manifold.
Let $\Symp(M,\omega)$ denote the group of symplectomorphisms of $(M,\omega)$ with compact support and $\Symp_0(M,\omega)$ denote the identity component of $\Symp(M,\omega)$.
In this section, we endow $\Symp(M,\omega)$ with the $C^\infty$-topology.

For a smooth function $H\colon M\to\mathbb{R}$, we define the \textit{Hamiltonian vector field} $X_H$ associated with $H$ by
\[\omega(X_H,V)=-dH(V)\text{ for every }V \in \mathcal{X}(M),\]
where $\mathcal{X}(M)$ is the set of smooth vector fields on $M$.

For a smooth function $H\colon  [0,1] \times M\to\mathbb{R}$ with compact support and for $t \in  [0,1] $, we define a function $H_t\colon M\to\mathbb{R}$ by $H_t(x)=H(t,x)$.
Let $X_H^t$ denote the Hamiltonian vector field associated with $H_t$ and let $\{\varphi_H^t\}_{t\in\mathbb{R}}$ denote the isotopy generated by $X_H^t$ such that $\varphi^0=\mathrm{id}$.
We set $\varphi_H=\varphi_H^1$ and $\varphi_H$ is called the \emph{Hamiltonian diffeomorphism generated by $H$}. For a 
symplectic manifold $(M,\omega)$, we define the group $\Ham(M,\omega)$ of Hamiltonian diffeomorphisms by
\[\Ham(M,\omega)=\{\varphi\in\mathrm{Diff}(M)\;|\;\exists H\in C^\infty([0,1]\times M)\text{ such that }\varphi=\varphi_H\}.\]
Then, $\Ham(M,\omega)$ is a normal subgroup of $\Symp_0(M,\omega)$.

A smooth function $H\colon  [0,1] \times M\to\mathbb{R}$ is said to be \emph{normalized}  if
\begin{itemize}
\item $M$ is closed and $\int_MH_t\omega^n=0$ for every $t$, or
\item $M$ is open and $H$ has a compact support.
\end{itemize}

Let $\widetilde{\Symp_0}(M,\omega)$ denote the universal covering of $\Symp_0(M,\omega)$.
We define the (symplectic) flux homomorphism $\widetilde{\flux}_\omega\colon\widetilde{\Symp_0}(M,\omega)\to \HHH_c^{1}(M;\RR)$ by
\[\widetilde{\flux}_\omega([\{\psi^t\}_{t\in[0,1]}])=\int_0^1[\iota_{X_t}\omega]dt,\]
where $\{\psi^t\}_{t\in[0,1]}$ is a path in $\Symp_0(M,\omega)$ with $\psi^0=\id$ and $[\{\psi^t\}_{t\in[0,1]}]$ is the element of the universal covering $\widetilde{\Symp_0}(M,\omega)$ represented by the path $\{\psi^t\}_{t\in[0,1]}$.
It is known that $\widetilde{\flux}_\omega$ is a well-defined homomorphism (\cite{Ban}, \cite{Ban97}, \cite[Lemma 2.4.3]{O15a}).

\begin{prop}[\cite{Ban}, see also \cite{Ban97} and {\cite[Proposition 10.18]{MS}}]\label{survey on flux}
  Let $(M,\omega)$ be a closed 
  symplectic manifold.
  Then, the following hold.
   \begin{enumerate}[$(1)$]
  \item The map $\tHam(M,\omega) \to \tSymp(M,\omega)$ induced by the inclusion map $\Ham(M,\omega) \to \Symp(M,\omega)$ is an injective homomorphism.
  \item The flux homomorphism $\widetilde{\flux}_\omega\colon \widetilde{\Symp_0}(M,\omega)\to \HHH^1(M;\RR)$ is surjective.
  \item $\mathrm{Ker}(\widetilde{\flux}_\omega)=\widetilde{\Ham}(M,\omega)$.
\end{enumerate}
\end{prop}
By Proposition \ref{survey on flux}, we can regard $\tHam(M,\omega)$ as a subgroup of $\tSymp(M,\omega)$.
Moreover, we can confirm that $\tHam(M,\omega)$ is a normal subgroup of $\tSymp(M,\omega)$.

Next, we recall the Calabi homomorphism. 
Let $(M,\omega)$ be a $2n$-dimensional  open symplectic manifold.
The \textit{Calabi homomorphism}
is a function $\mathrm{Cal}_{M} \colon \tHam(M,\omega)\to\mathbb{R}$ defined by
\[
	\mathrm{Cal}_{M}(\varphi_H)=\int_0^1\left(\int_M H_t \omega^n\right)dt,
\]
 where $H \colon [0,1] \times M \to \RR$ is a normalized Hamiltonian function.
 It is known that the Calabi homomorphism is a well-defined group homomorphism (see \cite{Cala}, \cite{Ban}, \cite{Ban97}, \cite{MS}, \cite{Hum}, \cite[Theorem 2.5.6]{O15a}).
Since $\int_M H\omega^n=\int_M (H\circ\psi)\omega^n$ holds for every smooth function $H\colon M \to \RR$ and every $\psi \in \Symp(M,\omega)$,  the Calabi homomorphism is $\Symp(M,\omega)$-invariant.

Finally, we recall the first Chern class of a symplectic manifold.
Let $(M,\omega)$ be a symplectic manifold.
An almost complex structure $J$ is said to be \textit{$\omega$-compatible} if $\omega(\cdot, J\cdot)$ is a Riemannian metric on $M$.
The \textit{first Chern class} $c_1(M) \in \HHH^2(M)$ of $(M,\omega)$ is defined to be the first Chern class of $TM$ equipped with an $\omega$-compatible complex structure.
Since the space of  $\omega$-compatible complex structures is connected, $c_1(M)$ does not depend on the choice of $J$.


\subsection{Quasimorphisms}\label{subsec=qm}
We recall the definition and properties of quasimorphisms which we need in later sections. 
For the basics of quasimorphisms, we refer the reader to \cite{Ca}.
A real-valued function $\phi \colon G \to \RR$ on a group $G$ is called a {\it quasimorphism} if
\[
 D(\phi):= \sup_{g,h \in G}|\phi(gh) - \phi (g) - \phi(h)|
\]
 satisfies that $D(\phi)<\infty$.
The constant $D(\phi)$ is called the {\it defect} of $\phi$.
A quasimorphism $\phi$ is said to be {\it homogeneous} if
$\phi(g^m) = m \cdot \phi(g)$ for every $g \in G$ and for every $m \in \ZZ$.
The following properties are fundamental and well-known properties of homogeneous quasimorphisms. 

\begin{lem} \label{lem:qm}
  Let $\phi$ be a homogenous quasimorphism on a group $G$. Then, for every $g,h \in G$, the following hold true:
  \begin{enumerate}[$(1)$]
    \item $\phi(hgh^{-1})=\phi(g)$;
    \item if $gh=hg$, then $\phi(gh)=\phi(g)+\phi(h)$.
  \end{enumerate}
\end{lem}

 The homogeneity condition on quasimorphisms is not restrictive in the following sense.

\begin{lem}[homogenization of quasimorphisms]\label{lem=homog}
Let $G$  be a group and $\phi$ a quasimorphism on $G$, not necessarily homogeneous. Then there exists a unique quasimorphism $\phi_{\mathrm{h}}$ on $G$ satisfying the following two properties:
\begin{enumerate}[$(1)$]
  \item $\phi_{\mathrm{h}}$ is \emph{homogeneous}.
  \item $|\phi_{\mathrm{h}}(g)-\phi(g)| \leq D(\phi)$ for every $g\in G$.
\end{enumerate}
\end{lem}

The quasimorphism $\phi_{\mathrm{h}}$ is called the \emph{homogenization} of $\phi$.
This homogenization is explicitly given by 
\[
  \phi_{h}(g) = \lim_{n \to \infty} \frac{\phi(g^n)}{n}.
\] 

Secondly, we recall the definition of invariance of quasimorphisms for the pair $(\Gg,\Ng)$ of a group and its normal subgroup.

\begin{definition}[invariant quasimorphism]\label{defn=inv_qm}
Let $\Gg$ be a group and $\Ng$ a normal subgroup of $\Gg$.
A quasimorphism $\phi \colon \Ng \to \RR$ on $\Gg$ is said to be \emph{$\Gg$-invariant} if 
\[
\phi(g x g^{-1}) = \phi(x)
\]
holds for every $g \in \Gg$ and for every $x \in \Ng$.
\end{definition}

\begin{lem} \label{lem=inv}
Let $\Gg$ be a group and $\Ng$ a normal subgroup of $\Gg$.
\begin{enumerate}[$(1)$]
 \item Assume that $\phi\colon \Gg \to \RR$ is a homogeneous quasimorphism on $\Gg$. Then, $\phi$ is $\Gg$-invariant.
 \item Assume that $\psi\colon \Gg\to \RR$ is a homogeneous quasimorphism on $\Gg$. Then the restriction $\psi |_{\Ng}\colon \Ng\to \RR$ of $\psi$ to $\Ng$ is a $\Gg$-invariant homogeneous quasimorphism on $\Ng$.
\end{enumerate}
\end{lem}

\begin{proof}
  Item (1) follows from Lemma ~\ref{lem:qm}~(1). 
  Item (2) immediately follows from item (1).
\end{proof}

\begin{remark}\label{remark=homog}
Let $\Gg$ be  a group and $\Ng$ a normal subgroup of $\Gg$. Let $\phi\colon \Ng\to \RR$ be a $\Gg$-invariant homogeneous quasimorphism on $\Ng$. 
Suppose that $\phi$ is extendable to $\Gg$.
Then, we remark that there exists a \emph{homogeneous} extension of $\phi$. To see this, take the homogenization $\psi_{\mathrm{h}}$ of $\psi$. Since $\phi\colon \Ng\to \RR$ is homogeneous by assumption, the restriction of $\psi_{\mathrm{h}}$ to $\Ng$ coincides with that of $\psi$. Hence, $\psi_{\mathrm{h}}|_{\Ng}=\phi$. 
\end{remark}

We review a sufficient condition for an invariant quasimorphism to be extendable, which we will use in the proof of Lemma \ref{extendable}.

\begin{lem}[{see \cite[Remark 8.3]{IshidaThesis}, \cite[Proposition 3.1]{KK} for example}]\label{lem=extendable}
  If the short exact sequence $1\to \Ng \to \Gg \to \Gamg \to 1$ of groups splits, then every $\phi \in \QQQ(\Ng)^{\Gg}$ is extendable to $\Gg$.
\end{lem}

Next, we explain the descendibility of quasimorphisms.
Let $\pi\colon \bG\to \bar{\bG}$ be a surjective homomorphism between groups $\bG$ and $\bar{\bG}$.
A quasimorphism $\phi$ on $\bG$ is \emph{descendible} to $\bG$ if there exists a quasimorphism $\bar\phi$ on $\bar{\bG}$ such that $\bar\phi\circ\pi=\phi$.
This $\bar\phi$ is called the \emph{descended quasimorphism} of $\phi$.


\subsection{Group cohomology and bounded cohomology}\label{subsec=bdd_cohom}
In this subsection, we briefly recall the definitions and properties of group cohomology and bounded cohomology.
We refer the reader to \cite{brown82}, \cite{Ca}, \cite{Fr} and \cite{Monod} for further information on them.

For a group $G$, we define $\CCC^n(G)$ to be the $\RR$-vector space consisting of functions from the $n$-fold direct product $G^{\times n}$ to $\RR$ for $n \geq 1$ and $\CCC^0(G) = \RR$.
An element $c$ of $\CCC^n(G)$ is called an \emph{$n$-cochain on $G$}.
This $\{\CCC^n(G)\}_{n\geq 0}$ gives rise to a cochain complex with respect to the coboundary map $\delta \colon \CCC^n(G) \to \CCC^{n+1}(G)$ defined by
\begin{align*}
  &\big(\delta (c)\big) (h_1, \cdots, g_{n+1}) \\
  &= c(g_2, \cdots, g_{n+1}) + \sum_{i = 1}^{n}(-1)^{i} c(h_1, \cdots, g_ig_{i+1}, \cdots, g_{n+1}) + (-1)^{n+1} c(h_1, \cdots, g_n)
\end{align*}
for $n \geq 1$ and $\delta = 0$ for $n = 0$.
The homology of this cochain complex $(C^*(G), \delta)$ is called the \emph{group cohomology of $G$} and denoted by $\HHH^*(G)$.
By construction, the first cohomology group $\HHH^1(G)$ is equal to the space $\Hom(G,\RR)$ of homomorphisms from $G$ to $\RR$.

If $N$ is a normal subgroup of $G$, then the conjugation $G$-action on $\CCC^*(N)$ induces the $G$-action on $\HHH^*(N)$, and $\HHH^*(N)^G$ denotes the invariant part of the $G$-action.

Let $\CCC_b^n(G)$ be the subset of $\CCC^*(G)$ consisting of bounded functions.
An element of $\CCC_b^n(G)$ is called a \emph{bounded $n$-cochain on $G$}.
We can easily confirm that $(\CCC_b^*(G), \delta)$ is a subcomplex of $(\CCC^*(G), \delta)$.
The homology of the subcomplex $(\CCC_b^*(G), \delta)$ is called the \emph{bounded cohomology of $G$} and denoted by $\HHH_b^*(G)$.
Gromov showed that if $G$ is amenable (in particular, abelian), then the bounded cohomology group $\HHH_b^n(G)$ is trivial for $n \geq 1$.
For a normal subgroup $N$ of $G$, the $G$-action on $\HHH_b^*(N)$ is defined in the same way as in the case of group cohomology.

Let $\HHH_{/b}^*(G)$ be the homology of the relative complex $(\CCC^*(G)/\CCC_b^*(G), \delta)$.
Then the short exact sequence
\[
  0 \to \CCC_b^*(G) \to \CCC^*(G) \to \CCC^*(G)/\CCC_b^*(G) \to 0
\]
induces the cohomology long exact sequence
\begin{align*}
  \cdots \to \HHH_b^n(G) \xrightarrow{c_G} \HHH^n(G) \to \HHH_{/b}^n(G) \to \HHH_b^{n+1}(G) \to \cdots.
\end{align*}
The map $c_G \colon \HHH_b^n(G) \to \HHH^n(G)$ is called the \emph{comparison map}.
Together with the fact that $\HHH_b^1(G) = 0$ and the cohomology $\HHH_{/b}^1(G)$ is isomorphic to the space $\QQQ(G)$ of homogeneous quasimorphisms on $G$, this long exact sequence yields the exact sequence
\[
  0 \to \HHH^1(G) \to \QQQ(G) \xrightarrow{\mathbf{d}} \HHH_b^2(G) \xrightarrow{c_G} \HHH^2(G).
\]
Here the map $\mathbf{d} \colon \QQQ(G) \to \HHH_b^2(G)$ is given by $\mathbf{d}(\phi) = [\delta \phi]$.

An exact sequence $1 \to N \to G \to \GG \to 1$ induces the following exact sequences of group cohomology, which is called the \emph{five-term exact sequence}:
\begin{align*}
  0 \to \HHH^1(\GG) \to \HHH^1(G) \to \HHH^1(N)^G \to \HHH^2(\GG) \to \HHH^2(G).
\end{align*}
Moreover, if $\HHH^k(N) = 0$ for $1 \leq k \leq n-1$, then the following exact sequence exists (\cite{hochschild_serre53}):

\begin{align}\label{5term}
  0 \to \HHH^n(\GG) \to \HHH^n(G) \to \HHH^n(N)^G \to \HHH^{n+1}(\GG) \to \HHH^{n+1}(G).
\end{align}

There is a bounded cohomology version of the five-term exact sequence as follows (\cite{Monod}):
\begin{align}\label{5term_bdd}
  0 \to \HHH_b^2(\GG) \to \HHH_b^2(G) \to \HHH_b^2(N)^G \to \HHH_b^3(\GG) \to \HHH_b^3(G).
\end{align}


\subsection{Shelukhin's quasimorphism} \label{subsec:sh_qm}
For a $2n$-dimensional symplectic manifold $(M^{2n},\omega)$ of finite volume, Shelukhin constructed a homogeneous quasimorphism $\mathfrak{S}_M$ on $\tHam(M, \omega)$ \cite{Shelukhin}.
In this subsection, we briefly recall the definition and properties of Shelukhin's quasimorphism $\mathfrak{S}_M$.

Let $\mathcal{S} \to M$ be the fiber bundle whose fiber over $x \in M$ is the space of $\omega_x$-compatible complex structures on the tangent space $T_x M$.
Let $\mathcal{J}(M,\omega)$ be the space of global sections of $\mathcal{S} \to M$, that is, the space of $\omega$-compatible almost complex structures on $M$.
Hereafter, we abbreviate $\mathcal{J}(M,\omega)$ as $\mathcal{J}$.
Note that the space $\mathcal{J}$ is contractible since the fibers of $\mathcal{S} \to M$ are homeomorphic to the contractible space $\Sp(2n, \RR)/\UU(n)$.
The canonical $\Sp(2n, \RR)$-invariant K\"{a}hler form, which is induced from the trace, defines the fiberwise-K\"{a}hler form $\sigma$ on $\mathcal{J}$.
If $M$ is closed, then a form $\Omega$ on $\mathcal{J}$ is defined by $\Omega(A, B) := \int_M\sigma_x(A_x, B_x) \omega^n(x)$.
Note that the $\Ham(M, \omega)$-action on $\mathcal{J}$ preserves $\Omega$ (and indeed, so does the $\Symp(M, \omega)$-action).
We also note that this fiberwise-K\"{a}hler form $\sigma$ induces the metric structure on $\mathcal{J}$.
For every two points in $\mathcal{J}$, the geodesic between the two points is uniquely determined.

\begin{remark}
If $M$ is not closed, we need a little care. In fact, in this case, $\sigma_x(A_x, B_x)$ may not be integrable with respect to $\omega^n$.
However, since we consider the Hamiltonian diffeomorphisms with compact supports, the subsequent construction is verified for (not necessarily closed) symplectic manifolds of finite volume. See \cite[Remark 1.7.5]{Shelukhin}.
\end{remark}

The $\Ham(M,\omega)$-action on $(\mathcal{J}, \Omega)$ is Hamiltonian (\cite{MR1622931}, \cite{MR1207204}), and its moment map $\mu$ is given as follows.
For an $\omega$-compatible almost complex structure $J \in \mathcal{J}$, let $h$ be the Hermitian metric induced from $J$ and $\omega$.
Let $\nabla$ be the second canonical connection of $(M,J,h)$, that is, the connection uniquely defined by the three conditions that $\nabla J = 0$, $\nabla h = 0$, and the $(1,1)$-part of $\nabla$ is everywhere vanishing (\cite[Section 2.6]{MR1456265}).
Let $\rho \in \Omega^2(M;\RR)$ be the $i$ times the curvature of the connection of the Hermitian line bundle $\Lambda_{\CC}^n(TM)$ induced from $\nabla$.
Then the Hermitian scalar curvature $S(J) \in C^{\infty}(M,\RR)$ is defined by
\[
  S(J)\omega^n = n\rho\wedge\omega^{n-1},
\]
and the moment map $\mu \colon C_{\textrm{normal}}^{\infty}(M,\RR) \to C^{\infty}(\mathcal{J},\RR)$ is given by
\[
  (\mu(H))(J) = \int_{M}S(J)H\omega^n.
\]
Here $C_{\textrm{normal}}^{\infty}(M,\RR)$, which is the Lie algebra of $\Ham(M,\omega)$, is the space of normalized Hamiltonian functions.

\begin{definition}
  For a symplectic manifold $(M,\omega)$ of finite volume, we define the average Hermitian scalar curvature $A(M,\omega)$ by
\begin{equation}\label{eq:aHsc}
A(M,\omega)= \left. \int_{M} S(J) \omega^n \middle/ \int_{M} \omega^{n} \right.
= \left. n \int_{M} c_1(M) \omega^{n-1} \middle/ \int_{M} \omega^{n}. \right.
\end{equation}
\end{definition}

The geodesic from $J_0$ to $J_1$ will be denoted by $[J_0, J_1]$.
For every $J_0, J_1, J_2$ in $\mathcal{J}$, let $\Delta(J_0, J_1, J_2)$ denote the $2$-simplex whose restriction to each fiber $\mathcal{S}_x$ over $x \in M$ is the geodesic triangle $\Delta((J_0)_x, (J_1)_x, (J_2)_x)$ with respect to the K\"{a}hler form $\sigma_x$ with vertices $(J_0)_x$, $(J_1)_x$, and $(J_2)_x$.

We are now ready to define Shelukhin's quasimorphism $\mathfrak{S}_M$ on $\tHam(M,\omega)$.
For $\tg \in \tHam(M,\omega)$, let $\{ h_t \}_{t \in [0,1]}$ be a path from $h_0 = \id_M$ to $h_1 = g$ in $\Ham(M,\omega)$ representing $\tg$.
Let $J \in \mathcal{J}$ be a basepoint and $D$ a disk in $\mathcal{J}$ whose boundary is equal to the loop $\{ h_t \cdot J \}_{t \in [0,1]} \ast [g\cdot J, J]$, where $\ast$ denotes concatenation.
Let $\{ H_t \}_{t \in [0,1]}$ be the path in $C_{\textrm{normal}}^{\infty}(M,\RR)$ corresponding to the path $\{ h_t \}_{t \in [0,1]}$.
Then we define $\nu_{J} \colon \tHam(M,\omega) \to \RR$ by
\[
  \nu_J(\tg) := \int_D \Omega - \int_0^1 \mu(H_t)(h_t\cdot J) dt.
\]
This function $\nu_{J}$ is a quasimorphism since the equality
\begin{align}\label{bounded_coboundary}
  \nu_J(\widetilde{g} \widetilde{h}) - \nu_J(\widetilde{g}) - \nu_J(\widetilde{h}) = \int_{\Delta(J, gJ, ghJ)} \Omega
\end{align}
and the inequality
\begin{align}\label{bounded_R}
  \left| \int_{\Delta(J, gJ, ghJ)} \Omega \right|
  = \left| \int_M \middle( \int_{\Delta((J_x), (gJ)_x, (ghJ)_x)} \sigma_x \middle) \omega^n \right| \leq \Vol(M,\omega^n) \cdot C
\end{align}
hold, where $C$ is the area of an ideal triangle of $\Sp(2n,\RR)/\UU(n)$ with respect to the canonical K\"{a}hler form.

\begin{definition}[{\cite{Shelukhin}}]
  The homogeneous quasimorphism
  \[
    \mathfrak{S}_M \colon \tHam(M,\omega) \to \RR
  \]
  is defined as the homogenization of $\nu_J$.
\end{definition}

Note that the homogenization $\mathfrak{S}_M$ of $\nu_{J}$ does not depend on the choice of $J$ (\cite[Theorem 1]{Shelukhin}).

Shelukhin \cite[Corollary 2]{Shelukhin} proved that the restriction of $\sqm_M$ to $\pi_1(\Ham(M,\omega))$ coincides with the homomorphism $I_{c_1}$ introduced in \cite{MR1666763}, which is defined as follows.
Note that $\pi_1(\Ham(M,\omega))$ is bijective to the set of isomorphism classes of fiber bundles $P \to S^2$ whose fiber is $M$ and whose structure group is contained in $\Ham(M,\omega)$.
For $\gamma \in \pi_1(\Ham(M,\omega))$, let $P_{\gamma} \to S^2$ denote the corresponding fiber bundle.
Note that the vertical tangent bundle $T_{V} P \to P$ is a symplectic vector bundle.
Let $c_1^{V}$ be the first Chern class of $T_{V} P \to P$.
There is another characteristic class $u \in \HHH^2(P;\RR)$, which is defined as a cohomology class satisfying $u|_{\text{fiber}} = [\omega]$ and $\int_{\text{fiber}} u^{n+1} = 0$.
Then $I_{c_1}(\gamma)$ is defined as
\[
  I_{c_1}(\gamma) = \int_{P_{\gamma}} c_1^{V} u^n.
\]
This gives rise to a homomorphism $I_{c_1} \colon \pi_1(\Ham(M,\omega)) \to \RR$.

The following property of $\sqm_M$ is a key to the proof of the non-extendability of Shelukhin's quasimorphism (Theorem~\ref{main_non_ext}).


\begin{prop} [\cite{Shelukhin}] \label{embedding functoriality}
  Let $U$ be an open subset of a closed symplectic manifold $(M,\omega)$.
  We write $\mathfrak{S}_M$ and $\mathfrak{S}_U$ for the Shelukhin quasimorphisms for $(M,\omega)$ and $(U,\omega |_U)$, respectively. 
  Then,
    \[ \mathfrak{S}_M \circ \iota_\ast = \mathfrak{S}_U - A(M,\omega) \cdot \mathrm{Cal}_U,\]
 where $\iota_\ast \colon \tHam(U) \to \tHam(M)$ is the natural homomorphism induced by the inclusion $\iota \colon U \to M$.
  \end{prop}



\subsection{Reznikov class}\label{subsec:re_class}
Here, we recall the definition of the Reznikov class $R$.
Let $(M, \omega)$, $\Omega$, $J$, $\Delta(J_0, J_1, J_2)$ be as in Subsection \ref{subsec:sh_qm}.
For $g,h \in \Symp(M,\omega)$, we define a two-cocycle $b_J$ on $\Symp(M,\omega)$ by 
\[
  b_J (g,h) = \int_{\Delta(J, gJ, ghJ)} \Omega.
\]

\begin{definition}\label{def:Reznikov}
  The \textit{Reznikov class} is the group cohomology class $R = [b_J] \in \HHH^2(\Symp(M,\omega))$.
  The \textit{bounded Reznikov class} is the bounded cohomology class $R_b = [b_J] \in \HHH_b^2(\Symp(M,\omega))$.
  Let $R_0 \in \HHH^2(\Symp_0(M,\omega))$ (resp. $R_{b,0} \in \HHH_b^2(\Symp_0(M,\omega))$) be the Reznikov class (resp. the bounded Reznikov class) restricted to $\Symp_0(M,\omega)$.
\end{definition}

\begin{remark}\label{rem:indep}
  As their symbols suggest, the Reznikov class (resp. the bounded Reznikov class) does not depend on the choice of $J$.
  Indeed, for $J_0, J_1 \in \mathcal{J}$, the $1$-cochain $a \in \CCC_b^1(\Symp(M,\omega))$ defined by
  \[
    a(g) = \int_{\Delta(J_0, J_1, gJ_1) + \Delta(J_0, gJ_1, gJ_0)} \Omega
  \]
  satisfies $b_{J_1} - b_{J_0} = \delta a$.
\end{remark}

As we mentioned in the introduction, the pullback of the Reznikov class to $\tHam(M, \omega)$ is trivial. This follows from \eqref{bounded_coboundary}.



\subsection{Symplectic blow-up}\label{subsec:blowup}

In this subsection, we review a blow-up of a symplectic manifold. Throughout this subsection, we always assume that the dimension of a symplectic manifold is larger than 2. 
First we review a blow-up of a topological manifold.

We set 
\[
  B^{2n}(R) = \{ v = (z_1, \cdots, z_n) \in \CC^{n} \mid \| v \| < R \},
\] 
where $\| v\| = \sqrt{|z_1|^2 + \cdots + |z_n|^2}$. 
Let $M$ be a $2n$-dimensional topological manifold and $\iota \colon B^{2n}(R) \to M$ a topological embedding.
We set $x_0=\iota(0)$.
For every real number $r$ with $ 0 <  r<R$, we set
\[U_r=\{
(v,[w]) \in \CC^n \times \CC P^{n-1}\, | \, w\in \CC^n\setminus\{0\}, v\in \CC w, \|v \| < r
\}.\]
Here, $[w]$ denotes the image of $w$ through the natural projection $\CC^n\setminus\{0\} \to \CC P^{n-1}$. 
Then, we define the \emph{blow-up} $\hat{M}_\iota$ of $M$ (with respect to $\iota$) to  be
\[\hat{M}_\iota=(M\setminus \{x_0\}) \sqcup U_r/\sim,\]
where the equivalence relation $\sim$ identifies a point $(z,[w]) \in U_r$ with $\iota(z) \in M$ for each $z\neq0$ .

Let $\iotap \colon M \setminus \{ x_0\} \to \hat{M}_\iota$ and $\iotau \colon U_r \to \hat{M}_\iota$ be the inclusions.
More precisely, $\iotap$ is the composition of the map $M \setminus \{ x_0\} \to (M \setminus \{ x_0\}) \sqcup U_r$ and the quotient map $(M \setminus \{ x_0\}) \sqcup U_r \to \hat{M}_\iota$, and $\iotau$ is the composition of the map $U_r \to (M \setminus \{ x_0\}) \sqcup U_r$ and the quotient map $(M \setminus \{ x_0\}) \sqcup U_r \to \hat{M}_\iota$, respectively.

We define the map $\hat\iota \colon \CC P^{n-1} \to \hat{M}_\iota$ by
\[\hat\iota ([w])=\iotau(0,[w]).\]
Let $\pi \colon \hat{M}_\iota \to M$ be the projection.
Then, the following proposition is known.
\begin{prop}[for example, see {\cite[Proposition 9.3.3]{MS12}}]\label{blow_up_def}
Let $(M,\omega)$ be a symplectic manifold, $\iota \colon B^{2n}(r) \to M$ a symplectic embedding, and $J$ an $\omega$-compatible almost complex structure such that $\iota^\ast J=J_0$, where $J_0$ is the standard complex structure on $B^{2n}(r)$.
Then, for every $ 0 < \rho < r$, there exists a symplectic form $\omega_\rho$ on the blow-up $\hat{M}_\iota$ with the following properties.
\begin{enumerate}[$(1)$]
 \item The $2$-form $\pi^\ast \omega$ corresponds to $\omega_\rho$ on $\pi^{-1}(M\setminus \iota(B^{2n}(r)))$.
 \item $\hat{\iota}^\ast \omega_\rho=\rho^2\omega_{FS}$, where $\omega_{FS}$ is the Fubini--Study form on $\CC P^n$.
\end{enumerate}
Here, under our convention, we set $\omega_{FS}([\CC P^1])=1$.
\end{prop}

The symplectic manifold $(\hat{M}_\iota, \omega_\rho)$ appearing in Proposition \ref{blow_up_def} is called the \emph{blow-up of $(M,\omega)$} (with respect to $\iota$, $J$ and $\rho$).


Let $r_1,\ldots.r_n$ be real numbers with $0<r_1<r_2<\ldots<r_n$.
Then, let $T^{2n}(r_1,\ldots,r_n)$ denote the torus $(\RR/2r_1\ZZ)^2 \times \cdots \times (\RR/2r_n\ZZ)^2$ with coordinates $(x_1,y_1,x_2,y_2, \cdots,x_n,y_n)$ and $J_0$ be the standard complex structure of $T^{2n}(r_1,\ldots,r_n)$.
We define the standard symplectic form $\omega_0$ of $T^{2n}(r_1, \ldots, r_n)$ by $\omega_0=dx_1\wedge dy_1+\cdots+dx_n\wedge dy_n$.
For a real number $r$ with $ 0 <  r < r_1$, let $\iota(r) \colon B^{2n}(r) \to T^{2n}(r_1,\ldots,r_n)$ denote an embedding defined by $\iota(r)(x_1,y_1,x_2, y_2, \cdots,x_n,y_n) = (x_1,y_1,x_2, y_2\cdots,x_n,y_n)$.

We have defined the blow-up of a symplectic manifold, and now we are ready to present the precise statements of our main theorems (Theorem~\ref{main_non_ext} and Theorem~\ref{reznikov_main}) 
for the case of the blow-up of a torus. 

\begin{thm}[main theorems for the case of the blow-up of a torus]\label{torus blow up}
 Let $n$ be a positive integer and $r_1,\ldots, r_n$ real numbers with $ 0 < r_1<r_2<\ldots<r_n$.
For real numbers $\rho,r$ with $ 0 < \rho<r < r_1$,
let $(\hat{M}_r,\omega_\rho)$ be the blow-up of $(T^{2n}(r_1,\ldots,r_n),\omega_0)$ with respect to $\iota(r)$, $J_0$ and $\rho$.
Then, the following hold.
\begin{enumerate}[$(1)$]
 \item  Shelukhin's quasimorphism $\sqm_{\hat{M}_r} \colon \tHam(\hat{M}_r,\omega_\rho) \to \RR$ is not extendable to $\tSymp(\hat{M}_r,\omega_\rho)$.
 \item The restriction $R|_{\Ham(\hat{M}_r,\omega_\rho)} \in \HHH^2(\Ham(\hat{M}_r,\omega_\rho))$ of the Reznikov class $R$ to $\Ham(\hat{M}_r,\omega_\rho)$
is non-zero.
In particular, the Reznikov class $R \in \HHH^2(\Symp(\hat{M}_r,\omega_\rho);\RR)$ is non-zero.
\end{enumerate}
\end{thm}

In the proof of Theorem \ref{torus blow up}, we use the following lemma.
Recall from Subsection \ref{subsec:sh_qm} that $A(M,\omega)$ is the average Hermitian scalar curvature.

\begin{lem}\label{torus blow up lemma}
Let $r_1,\ldots, r_n$ be real numbers with $0<r_1<r_2<\cdots<r_n$.
For real numbers $\rho,r$ with $ 0 < \rho<r < r_1$,
let $(\hat{M}_r,\omega_\rho)$ be a blow-up of $(T^{2n}(r_1,\ldots,r_n),\omega_0)$ with respect to $\iota(r)$, $J_0$ and $\rho$.
Then, $A(\hat{M}_r,\omega_\rho) \neq 0$.
\end{lem}
We will prove Lemma \ref{torus blow up lemma} in Section \ref{sec_blow_up}.


\section{Average Hermitian scalar curvature of the blow-up of a torus}\label{sec_blow_up}

In this section,  we discuss the case of blow-ups of a torus, indicated in Subsection~\ref{subsec:blowup}. Our main goal of this section is to prove Lemma \ref{torus blow up lemma}.
Before proving Lemma \ref{torus blow up lemma}, we review some basic properties of the first Chern class of a symplectic manifold.
For a symplectic manifold $(M,\omega)$, let $c_1(M) \in \HHH^2(M)$ denote its first Chern class.
As we recalled in Subsection \ref{subsec=symp}, the first Chern class can be defined for a symplectic manifold.
Then, the following properties of the first Chern class are well known.

\begin{prop} \label{basic chern}
Let $X$ be a symplectic submanifold of a symplectic manifold $(M,\omega)$.
Then, $c_1(X)=c_1(M)|_X$.
\end{prop}

Since the standard complex structure of $\CC P^n$ is compatible with the Fubini--Study form $\omega_{FS}$, 
the first Chern class of $(\CC P^n,\omega_{FS})$ coincides with the first Chern class of (the tangent bundle of) $\CC P^n$ with the standard complex structure, and hence it is non-zero.

Here, we start the proof of Lemma \ref{torus blow up lemma}.
Let $n\geq2$ and $r_1,\ldots,r_n$ be real numbers with $0<r_1<r_2<\cdots<r_n$.
For real numbers $\rho,r$ with $ 0 < \rho<r < r_1$,
let $(\hat{M}_r,\omega_\rho)$ be a blow-up of $(T^{2n}(r_1,\ldots,r_n),\omega_0)$ with respect to $\iota(r)$ , $J_0$ and $\rho$.
Set $(M,\omega)=(T^{2n}(r_1,\ldots,r_n),\omega_0)$ and let $(\hat{M}_r,\omega)$ be the blow-up of $(T^{2n}(r_1,\ldots,r_n),\omega_0)$ with respect to $\iota(r)$ , $J_0$ and $\rho$.

By the Mayer--Vietoris exact sequence, for $i = 1,2, \ldots, 2n-1$, we have the isomorphisms
\[ \HHH^i(\hat{M}_r) \xrightarrow{\cong} \HHH^i(\iotap(M \setminus \{x_0\})) \oplus \HHH^i(\iotau(U_r)) \xrightarrow{\cong} \HHH^i(M \setminus \{x_0\}) \oplus \HHH^i(U_r).\]
By the cohomology long exact sequence of the pair $(M, M \setminus \{x_0\})$, for $i = 1,2, \cdots, 2n-1$, we have an isomorphism
\[ \HHH^i(M \setminus \{x_0\}) \cong \HHH^i(M).\]
Since the natural inclusion $\CC P^{n-1} \to U_r$ has a deformation retract, we have an isomorphism
\[\HHH^i(U_r) \cong \HHH^i(\CC P^{n-1}).\]
Thus, we have the following lemma:

\begin{lem} \label{blow-up lemma 1}
For $i = 1, 2, \ldots, 2n-1$, there exist the following isomorphisms:
\[ \HHH^i(\hat{M}_r) \cong \HHH^i(\iotap(M \setminus \{x_0\})) \oplus \HHH^i(\iotau(U_r)) \cong \HHH^i(M) \oplus \HHH^i(\CC P^{n-1}).\]
\end{lem}

We also use the following lemma.
For the reader's convenience, we include the proof.
\begin{lem} \label{blow-up lemma 2}
Assume that $i, j \in \{ 1,2, \ldots, 2n-1\}$. For $\alpha \in \HHH^i(M)$ and $\beta \in \HHH^j(\CC P^{n-1})$, under the identification in Lemma \textup{\ref{blow-up lemma 1}}, we have
\[ \alpha \smile \beta = 0 \quad \textrm{ in } \; \HHH^{i+j}(\hat{M}_r).\]
\end{lem}
\begin{proof}
Consider the following exact sequence
\[ \HHH^i(\hat{M}_r, \iotau(U_r)) \to \HHH^i(\hat{M}_r) \to \HHH^i(\iotau(U_r)).\]
Since the map $\HHH^i(\hat{M}_r) \to \HHH^i(\iotau(U_r))$ sends $\alpha$ to $0$, there exists $v \in \HHH^i(\hat{M}_r, \iotau(U_r))$ such that $v$ is mapped to $\alpha$ by the map $\HHH^i(\hat{M}_r, \iotau(U_r)) \to \HHH^i(\hat{M}_r)$. Similarly, there exists $w \in \HHH^j(\hat{M}_r, \iotap(M \setminus \{x_0\}))$ such that $w$ is mapped to $\beta$ by the map $\HHH^j(\hat{M}_r, \iotap(M \setminus \{x_0\})) \to \HHH^j(\hat{M}_r)$. Then $v \smile w \in \HHH^{i + j}(\hat{M}_r, \hat{M}_r) = 0$ is mapped to $\alpha \smile \beta \in \HHH^{i+j}(\hat{M}_r)$, which implies $\alpha \smile \beta = 0$.
\end{proof}

Now we proceed to prove Lemma \ref{torus blow up lemma}.

\begin{proof}[Proof of Lemma~\textup{\ref{torus blow up lemma}}]
It suffices to prove that $c_1(\hat{M}_r) \smile [\omega_\rho]^{n-1} \ne 0$.
By Lemma~\ref{blow_up_def}(2), the restriction of $c_1(\hat{M}_r)$ to $\CC P^{n-1}$ coincides with the first Chern class of $\CC P^{n-1}$ with the standard complex structure.
As we noted above, $[\omega_{FS}] \neq 0$ on $\HHH^2(\CC P^{n-1})$ and thus  $c_1(\hat{M}_r) \neq 0$ on $\HHH^2(\hat{M}_r)$.
We also note that $\HHH^{2n-2}(\CC P^{n-1})\cong \RR$ and $[\omega_{FS}]^{n-1} \neq 0$ on $\HHH^{2n-2}(\CC P^{n-1})$.
Since $(M,\omega)=(T^{2n}(r_1,\ldots,r_n),\omega_0)$, the first Chern class $c_1(M)$ of $(M,\omega)$ vanishes.
Hence, by Lemmas \ref{blow-up lemma 1} and \ref{blow-up lemma 2}, if $c_1(\hat{M}_r) \smile [\omega_\rho]^{n-1} = 0$, then $c_1(\hat{M}_r) \smile \alpha = 0$ for every $\alpha \in \HHH^{2n-2}(\hat{M}_r)$.
This contradicts the Poincar\'e duality and thus we have completed the proof of Lemma \ref{torus blow up lemma}.
\end{proof}

\section{Proof of Theorem~\ref{main_non_ext}}\label{sec:proof_main}

\subsection{ Recapitulation of our constructions of symplectomorphisms}


In order to prove Theorem~\ref{main_non_ext}, we start from a local argument on the punctured $2$-torus.
We set $D_{\ee}=([0,1]\times[0,1])\setminus([2\ee, 1-2\ee] \times [2\ee, 1-2\ee])$ for fixed $\ee \in (0,1/4)$.
  Let $p\colon\RR^2\to\RR^2/\ZZ^2$ be the natural projection
  and set $P_{\ee}=p(D_{\ee})$.
  We consider $P_{\ee}$ as a  symplectic manifold with the symplectic form $\omega_0 = dx \wedge dy$, where $(x,y)$ is the standard coordinates on $P_{\ee} \subset \RR^2 / \ZZ^2$.

For $\qd =(a,b)\in \mathbb{R}^2$ satisfying $|b|\leq \ee$, we will construct vector fields $X_a$, $Y_b$ and ${\sigma}_{\qd}, {\tau}_{\qd} \in \widetilde{\Symp_0}(P_\ee)$.
Our construction comes from \cite{KK}, but here we use a notation similar to \cite{KKMM}.

We define the map $\tau_{\qd}$ in the following manner.
For every real number $b$ with $|b| \leq \ee$, there exists a vector field $Y_b$ on $P_\epsilon$ with compact support such that
\begin{itemize}
  \item $\mathcal{L}_{Y_b}\omega_0=0$,
  \item $(Y_b)_{p(x,y)}= b \frac{\partial}{\partial x}$ for every $(x,y)\in ([-\ee ,\ee ]\times\RR)\cup(\RR\times [-\ee,\ee])$
\end{itemize}
Here, $\mathcal{L}_{Y_b}$ is the Lie derivative with respect to $Y_b$.
Let $\tau_{\qd}^t$ be the time-$t$ map of the flow generated by $Y_b$.
Since $\mathcal{L}_{Y_b}\omega_0=0$,
we have $\tau_{\qd}^t \in \Symp_0(P_\epsilon,\omega_0)$.
Then, we define $\tau_{\qd}=[\{ \tau_{\qd}^t \}_{t \in [0,1]}] \in \tSymp(P_{\epsilon},\omega_0)$.




 We review the constructions of $\sigma_{\qd}$ for the case of $b>0$.
For a real number $r$ with $0< r \leq \ee$, let $\rho_{r} \colon [-\ee ,\ee ] \to[-1,1]$ be a smooth function satisfying the following conditions:
\begin{enumerate}[(1)]
\item $\Supp(\rho_{r})\subset(-r,r )$;
\item $\rho_{r}(x)+\rho_{r}(x+ r)= 1$ for every $x\in (-r,0)$.
\end{enumerate}
For a real number $r$ with $0 <  r \leq \ee$, 
 let $\hat{H}_{r} \colon [-\frac12, \frac12] \times \RR \to\RR$ be the function defined by
\[
\hat{H}_{r}(x,y)=
\begin{cases}
  -\rho_{r}(x) & \text{if $|x| \le \ee$}, \\
  0 & \text{otherwise}. \quad
\end{cases}
\]
Then, $\hat{H}_{r}$ induces the smooth function $H_{r} \colon P_\epsilon\to\RR$ with compact support. 

 Our definition of $\sigma_{\qd}$ is as follows: 
For a real number $t$,
we define $\sigma^t_{\qd} \in\Symp_0(P_\epsilon,\omega_0)$ by
\begin{equation*}\label{fab}
\sigma^t_{\qd}(z)=
\begin{cases}
\varphi_{H_{b}}^{at}(z) & \text{if $z\in p \bigl((-b,0)\times\RR \bigr)$}, \\
z & \text{otherwise},
\end{cases}
\end{equation*}
and set $\sigma_{\qd}=[\{\sigma^t_{\qd}\}_{t\in[0,1]}] \in \tSymp(P_{\epsilon},\omega_0)$.
Then, the vector field $X_a$ that generates the flow $\{ \sigma_{\qd}^t \}_{t\in\RR}$ is as in Figure \ref{vec_field}.

\begin{figure}[htbp]
  \begin{minipage}[c]{0.45\hsize}
    \centering
    \begin{overpic}[width=7truecm]{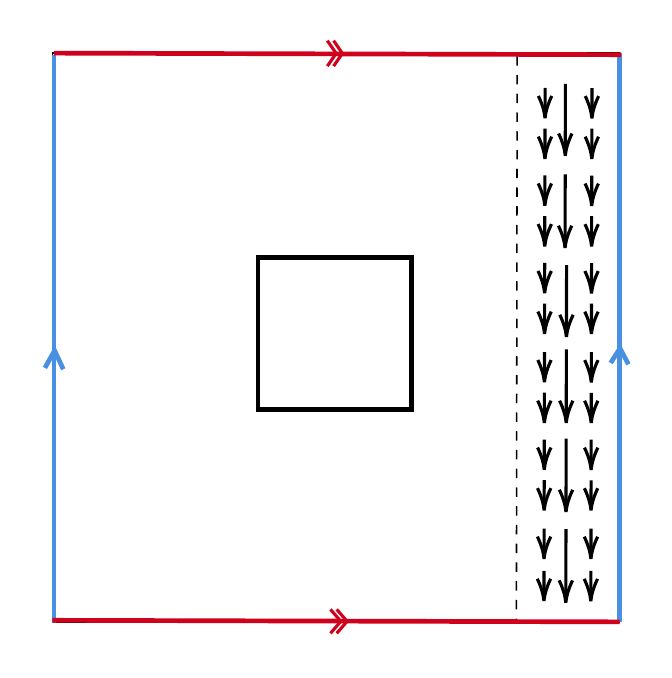}
      \put(46,-5){\LARGE{$X_a$} }
    \end{overpic}
    \vspace{1mm}
  \end{minipage}
  \hfill
  \begin{minipage}[c]{0.45\hsize}
    \centering
    \begin{overpic}[width=7truecm]{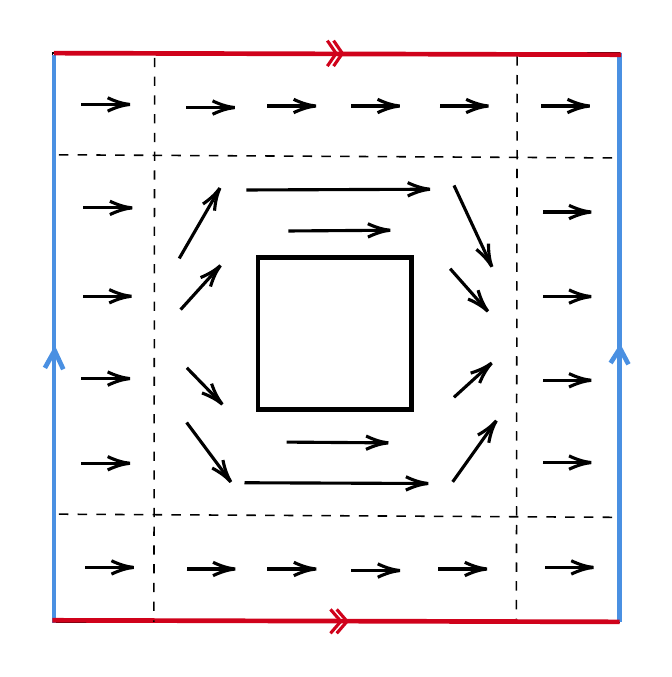}
      \put(46,-5){\LARGE{$Y_b$} }
    \end{overpic}
    \vspace{1mm}
  \end{minipage}
  \caption{The vector fields $X_a$ and $Y_b$ defining $\sigma_{\qd}$ and $\tau_{\qd}$ when $a>0$ and $b>0$. }
  \label{vec_field}
\end{figure}


Here we review several properties of ${\sigma}_{\qd}$ and ${\tau}_{\qd}$.
For proofs of them, see \cite{KK} or \cite{KKMM}.

\begin{lem}\label{lem:sigma_tau}
  Let $\qd =(a,b)\in \mathbb{R}^2$ satisfy $|b|\leq \ee$.
  Then ${\sigma}_{\qd}$ commutes with ${\tau}_{\qd} {\sigma}_{\qd}^{-1} {\tau}_{\qd}^{-1}$.
\end{lem}

   \begin{lem}\label{lem:calabi}
      Let $\qd =(a,b)\in \mathbb{R}^2$ satisfy $|b|\leq \ee$. Then, we have
      \[
        \mathrm{Cal}_{P_{\ee}}([\sigma_\qd , \tau_\qd])=-ab.
      \]
    \end{lem}

    \subsection{Basic calculation of  values of Shelukhin's quasimorphism}

\begin{lem} \label{average scalar curvature}
Let $(S,\omega_S)$ be a closed surface with a symplectic form $\omega_S$ and $(N, \omega_N)$ a closed symplectic manifold. Then,
  \[A (S \times N , \omega_{S\times N} ) = A(S,\omega_S) + A(N, \omega_N). \]
\end{lem}

\begin{proof}

Let $(2n-2)$ be the dimension of the manifold $N$. Since
\[(\omega_{S\times N})^{n} =n (\mathrm{pr}_1^{\ast} \omega_S) \wedge (\mathrm{pr}_2^{\ast} \omega_N)^{n-1}, \quad c_1(S\times N) = \mathrm{pr}_1^{\ast} c_1(S) + \mathrm{pr}_2^{\ast} c_1(N)\]
and
\[(\omega_{S\times N})^{n-1} =  (n-1) (\mathrm{pr}_1^\ast \omega_S) \wedge (\mathrm{pr}_2^{\ast} \omega_N)^{n-2}  + (\mathrm{pr}_2^{\ast} \omega_N)^{n-1},\]
we have
\begin{align*}
    & A (S\times N , \omega_{S\times N} )
    = n \left( \int_{S\times N} (\omega_{S\times N})^{n} \right)^{-1}\left( \int_{S\times N} c_1(S\times N) (\omega_{S\times N})^{n-1} \right) \\
    & = \left( \int_{N} \omega_N^{n-1} \right)^{-1} \left( \int_{S} \omega_S \right)^{-1}
    \left( \int_{S}  c_1(S) \int_{N}  \omega_N^{n-1} + (n-1) \int_S \omega_S \int_N c_1(N)  \omega_N^{n-2}   \right) \\
    & = \left( \int_{S} \omega_S \right)^{-1} \int_{S}  c_1(S) + (n-1) \left(\int_{N} \omega_N^{n-1}\right)^{-1} \int_N c_1(N)  \omega_N^{n-2} \\
    & =A(S,\omega_S) + A(N,\omega_N). \qedhere
  \end{align*}
\end{proof}

Let $T$ denote the $2$-torus with the standard symplectic form $\omega_0$.

\begin{lem} \label{extendable}
Let $(N, \omega_N)$ be a closed symplectic manifold.
Set $H= \Ham(T \times N, \omega_{T \times N} )$ and $G=\widetilde{\flux}^{-1}_{\omega_{T \times N}}(\HHH_c^1(T; \RR))$. 
Here we regard $\HHH_c^1(T; \RR)$ as a subspace of $\HHH^1_c(T \times N;\RR) \cong \HHH^1_c(T;\RR) \oplus \HHH^1_c(N;\RR)$. 
Then every $G$-invariant homogeneous quasimorphism on $H$ is extendable to $G$.
  \end{lem}

  \begin{proof}
    Since the exact sequence
\[1 \to H \to G \xrightarrow{\widetilde{\flux}_{\omega_{T \times N}}} \HHH_c^1(T; \RR) \to 1\]
 splits, it follows from Lemma \ref{lem=extendable} together with Remark \ref{remark=homog}.
  \end{proof}

\begin{lem} \label{lem:prod}
 Let $(N, \omega_N)$ be a $(2n-2)$-dimensional closed symplectic manifold.
Then
\[ \Cal_{P_{\ee}\times N} \circ \id^{P_\ee, P_{\ee}\times N}_\ast = n \cdot \Vol(N,\omega_N) \cdot \Cal_{P_{\ee}}.\]
Here, $\id^{P_\ee, P_{\ee}\times N}_\ast \colon \tHam(P_\ee, \omega_0) \to \tHam(P_\ee \times N, \omega)$ is the map defined by $\id=\id_{P_\ee \times N}$ as in Subsection $\ref{subsec:notation}$. 
   \end{lem}
\begin{proof}

Let $H \colon [0,1] \times P_\ee \to \RR$ be a normalized Hamiltonian function. 
If we define $H' \colon [0,1] \times (P_\ee \times N) \to \RR$ by $H'_t(x,y)=H_t(x)$, then $H'$ is normalized and $\id^{P_\ee, P_{\ee}\times N}_\ast(\phi_{H})=\phi_{H'}$.
Since $(\omega_{S\times N})^{n} =n (\mathrm{pr}_1^{\ast} \omega_S) \wedge (\mathrm{pr}_2^{\ast} \omega_N)^{n-1}$,
\begin{multline*}
  \Cal_{P_{\ee}\times N} \circ \id^{P_{\ee}, P_{\ee}\times N}_\ast(\phi_{H}) =\Cal_{P_{\ee}\times N}(\phi_{H'})
=\int_0^1 \left( \int_{P_{\ee}\times N} H'_t (\omega_{P_{\ee}\times N})^n \right) dt \\
= n \int_0^1 \left( \int_{P_{\ee}} H_t \omega_0 \int_{N} \omega_{N}^{n-1} \right) dt = n \cdot \Vol(N, \omega_N) \Cal_{P_{\ee}}(\phi_H). \qedhere
\end{multline*}

\end{proof}

\begin{prop} \label{prop:flux}
 Let $(N, \omega_N)$ and $(M, \omega)$ be closed symplectic manifolds of dimensions $2n-2$ and $2n$, respectively.
     Assume that  $\iPN \colon (P_{\ee} \times N,\omega_{P_{\ee}\times N})\to (M ,\omega)$ is an open symplectic embedding.
    Let $\qd =(a,b)\in \mathbb{R}^2$ satisfy $|b|\leq \ee$.
    Then
    \[
      \mathfrak{S}_{M} \left(\ioP^{P_\ee, M}_{\ast}\left([\sigma_{\qd},\tau_{\qd}]\right)\right)=-abn \cdot \Vol(N,\omega_N) \left( A(N,\omega_N) - A(M,\omega) \right).
    \]
    \end{prop}

    \begin{proof}
      Let $\kappa \colon P_\ee \times N \to T\times N$ be the inclusion map. Note that $\kappa$ is a symplectic embedding.
Set $\id=\id_{P_\ee \times N}$. By Proposition \ref{embedding functoriality},
\begin{multline*}
\mathfrak{S}_{M} \left( \ioP^{P_\ee, M}_{\ast}([{\sigma}_{\qd},{\tau}_{\qd}]) \right) \\
=   \mathfrak{S}_{P_\ee \times N}(\id^{P, P_\ee \times N }_{\ast}([{\sigma}_{\qd},{\tau}_{\qd}]))
        - A(M,\omega) \mathrm{Cal}_{P_\ee \times N}(\id^{P, P_\ee \times N }_{\ast}([{\sigma}_{\qd},{\tau}_{\qd}])). 
\end{multline*}
     By Proposition \ref{embedding functoriality}, Lemma \ref{average scalar curvature}, and since $A(T,\omega_0)=0$,
     \begin{multline*}
    \mathfrak{S}_{T \times N} \left( \kappa^{P_\ee, T \times N}_{\ast}([{\sigma}_{\qd},{\tau}_{\qd}]) \right) \\ =      \mathfrak{S}_{P_\ee \times N}( \id^{P_\ee,P_\ee \times N}_{\ast} [{\sigma}_{\qd},{\tau}_{\qd}])
      - A(N,\omega_{N}) \mathrm{Cal}_{P_\ee \times N}(\id^{P_\ee,P_\ee \times N}_{\ast}[{\sigma}_{\qd},{\tau}_{\qd}]).
      \end{multline*}
      By Lemma \ref{extendable}, $\sqm_{T \times N}$ extends to a homogeneous quasimorphism $\widehat{\mathfrak{S}}_{T \times N}$ on $\Flux^{-1}_{\omega_{T \times N}}(\HHH_c^1(T; \RR))$.
      Since $\kappa^{P_\ee, T \times N}_{\ast}({\sigma}_{\qd}), \kappa^{P_\ee, T \times N}_{\ast}( {\tau}_{\qd}) \in \Flux^{-1}_{\omega_{T \times N}}(\HHH_c^1(T; \RR))$,
      and by Lemmas \ref{lem:qm} and \ref{lem:sigma_tau},
    \begin{multline*}
      \mathfrak{S}_{T \times N} \left( \kappa^{P_\ee, T \times N}_{\ast}([{\sigma}_{\qd}, {\tau}_{\qd}]) \right) \\
      = \widehat{\mathfrak{S}}_{T \times N}(\kappa^{P_\ee, T \times N}_{\ast}({\sigma}_{\qd})) + \widehat{\mathfrak{S}}_{T \times N}(\kappa^{P_\ee, T \times N}_{\ast}( {\tau}_{\qd}{\sigma}_{\qd}^{-1}{\tau}_{\qd}^{-1} ))
      =0.
    \end{multline*}
    Hence we obtain
    \[\mathfrak{S}_{M} \left( \ioP_{\ast}^{P_{\ee},M}([{\sigma}_{\qd},{\tau}_{\qd}]) \right) =
    \left( A(N,\omega_N) - A(M,\omega) \right) \mathrm{Cal}_{P_\ee}\left(\id^{P_\ee,P_\ee \times N}_\ast( [{\sigma}_{\qd},{\tau}_{\qd}]) \right).\]
      Together with Lemmas \ref{lem:calabi} and \ref{lem:prod}, we complete the proof.
      \end{proof}

\begin{figure}[htbp]
    \centering
    \includegraphics[width=10truecm]{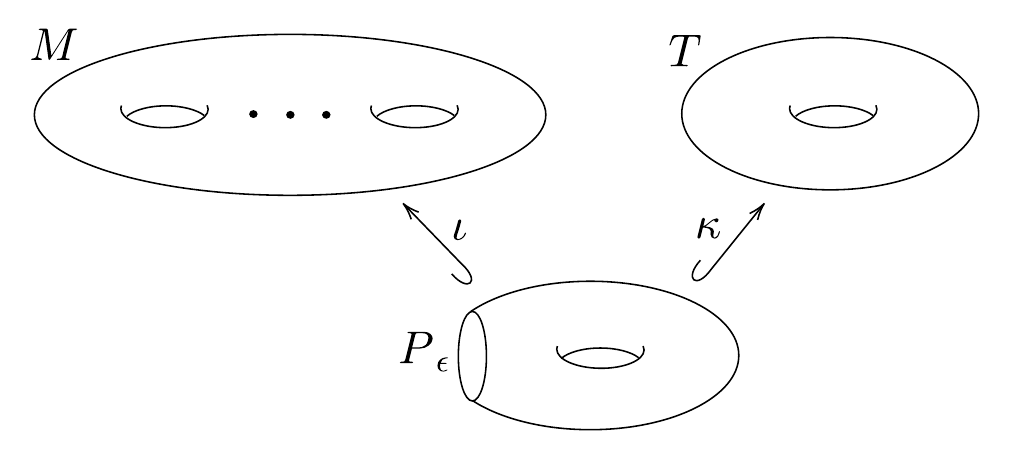}
    \caption{Idea of the proof of Proposition \ref{prop:flux} when $M$ is a closed surface of genus at least two and $N$ is the one-point set}
\end{figure}


\subsection{Proof of Theorem~\ref{main_non_ext}}

In this section, we prove Theorem~\ref{main_non_ext}.
Theorem~\ref{main_non_ext}(1) states that Shelukhin's quasimorphism $\sqm_M \colon \tHam(M,\omega) \to \RR$ is not extendable to $\tSymp(M,\omega)$ for the product $(M,\omega) = (S,\omega_S) \times (N,\omega_N)$, where $(S,\omega_S)$ is a closed 
orientable surface of genus at least two and $N$ is a closed symplectic manifold. 
Let $2n$ be the dimension of $M$.

\begin{proof}[Proof of Theorem~\textup{\ref{main_non_ext}(1)}]
  For a symplectic embedding $P_{\ee} \to S$, let $\iPN\colon (P_{\ee} \times N,\omega_{P_{\ee}\times N})\to (M ,\omega)$ be the induced symplectic embedding.
  By Proposition~\ref{prop:flux}, we have
  \[
    \mathfrak{S}_{M} \left(\ioP^{P_\ee, M}_{\ast}\left([\sigma_{\qd},\tau_{\qd}]\right)\right)=-abn \cdot \Vol(N,\omega_N) \left( A(N,\omega_N) - A(M,\omega) \right).
  \]
  Since $A(M,\omega) = A(S,\omega_S) + A(N,\omega_N)$ by Lemma~\ref{average scalar curvature}, we have
  \begin{align}\label{eq:sqm_nonzero}
    \sqm_M\left(\ioP^{P_\ee, M}_{\ast}\left([\sigma_{\qd},\tau_{\qd}]\right)\right) = abn \cdot \Vol(N, \omega_N) A(S,\omega_S) \neq 0.
  \end{align}
  Assume that $\sqm_M$ is extendable to $\tSymp(M,\omega)$. 
  Then by Lemma \ref{lem:sigma_tau} and the $\tSymp(M,\omega)$-invariance of $\sqm_M$, we have
  \[
    \sqm_M\left(\ioP^{P_\ee, M}_{\ast}\left([\sigma_{\qd},\tau_{\qd}]\right)\right) = \sqm_M\left(\ioP^{P_\ee, M}_{\ast}(\sigma_{\qd})\right) + \sqm_M\left(\ioP^{P_\ee, M}_{\ast}\left(\tau_{\qd}\sigma_{\qd}^{-1}\tau_{\qd}^{-1}\right)\right) = 0.
  \]
  This contradicts to \eqref{eq:sqm_nonzero}.
  Hence, $\sqm_M$ is not extendable to $\tSymp(M,\omega)$.
\end{proof}

Next we prove Theorem~\ref{main_non_ext}(2), which states that Shelukhin's quasimorphism $\sqm_{\hat{M}_r} \colon \tHam(\hat{M}_r,\omega_{\rho}) \to \RR$ is not extendable to $\tSymp(\hat{M}_r,\omega_{\rho})$ for the blow-up $\hat{M}_r$ of the symplectic torus $(T^{2n}, \omega)$ with $n \geq 2$.

\begin{proof}[Proof of Theorem~\textup{\ref{main_non_ext}(2)}]
  In fact, we prove Theorem~\ref{torus blow up}(1).
  For $\epsilon >0$ and a symplectic embedding $\kappa \colon P_{\epsilon} \to T^2(r_1)$, let 
  \[\hat\kappa \colon P_{\epsilon} \times T^{2(n-1)}(r_2,\dots, r_{n}) \to T^{2n}(r_1,\dots, r_n) \]
denote the symplectic embedding defined by 
\[\hat\kappa(z_1,z_2,\dots,z_n)=(\kappa(z_1),z_2,\dots,z_n).\]
Let $p_1\colon \RR^2\to T^{2}(r_1)$, $p_n\colon \RR^{2n}\to T^{2n}(r_1,\dots, r_n)$ be the standard projections.
Take $\epsilon>0$ such that the area of $P_\epsilon$ is smaller than the area of $T^2(r_1) \setminus p_1(B^2(r))$.
Then, we can take a symplectic embedding $\kappa \colon P_{\epsilon} \to T^2(r_1)$ such that the image of $\kappa$ and $p_1(B^2(r))$ are disjoint.
Thus, the image of $\hat\kappa \colon P_{\epsilon} \times T^{2(n-1)}(r_2,\dots, r_{n}) \to T^{2n}(r_1,\dots, r_n)$ and $p_n(B^{2n}(r))$ are disjoint.
Hence, $\hat\kappa$ induces the symplectic embedding
\[ \iota \colon P_{\epsilon} \times T^{2(n-1)}(r_2,\dots, r_{n}) \to \hat{M}_r.\]
   By Proposition~\ref{prop:flux}, we have
\begin{multline*}
  \sqm_{\hat{M}_r} \left(\ioP^{P_\ee, \hat{M}_r}_{\ast}\left([\sigma_{\qd},\tau_{\qd}]\right)\right) \\ 
  =-ab n \cdot \Vol(T^{2(n-1)}(r_2, \ldots, r_n), \omega_0) \left( A(T^{2(n-1)}(r_2, \ldots, r_n),\omega_0) - A(\hat{M}_r,\omega_{\rho}) \right).
\end{multline*}
  Since $A(T^{2(n-1)}(r_2, \ldots, r_n),\omega_0) = 0$ and $A(\hat{M}_r,\omega_{\rho}) \neq 0$ by Lemma~\ref{torus blow up lemma}, we have
  \[
    \sqm_M\left(\ioP^{P_\ee, \hat{M}_r}_{\ast}\left([\sigma_{\qd},\tau_{\qd}]\right)\right) = -ab n \cdot \Vol(T^{2(n-1)}(r_2, \ldots, r_n), \omega_0) \cdot A(\hat{M}_r,\omega_{\rho}) \neq 0.
  \]
  Then, by the argument similar to the proof of Theorem~\ref{main_non_ext}(1), $\sqm_M$ is not extendable to $\tSymp(M,\omega)$.
\end{proof}


\section{Applications to the Reznikov class and the proof of Theorem \ref{reznikov_main}}\label{reznikov sect}

\subsection{Non-triviality of the Reznikov class}
Recall from Subsection~\ref{subsec:re_class} the definition of the Reznikov class.
One of the ways to show the non-triviality of the Reznikov class is to pullback the class to a discrete group, such as $\Sp(2n,\ZZ)$ or mapping class groups.
For example, if $M$ is a $2n$-dimensional torus, then the Reznikov class restricts to the Borel class on $\Sp(2n,\ZZ)$, which is non-zero for large $n$.

In this section, we show the non-triviality of the Reznikov class from a different approach, which uses the non-extendability of Shelukhin's quasimorphism.

\begin{prop}\label{prop:non-zero_R}
  Let $(M, \omega)$ be a closed symplectic manifold.
  Assume that Shelukhin's quasimorphism $\mathfrak{S}_M$ on $\tHam(M, \omega)$ is not extendable to $\tSymp(M,\omega)$.
  Then $R_0 \in \HHH^2(\Symp_0(M,\omega))$ is non-zero. In particular, the Reznikov class $R \in \HHH^2(\Symp(M,\omega))$ is non-zero.
\end{prop}

Before proving the proposition, we discuss the vanishing of the Reznikov class in a more general setting.
Let $M$ be a closed symplectic manifold.
Let $\tL$ be a subgroup of $\tSymp(M,\omega)$ which contains $\tHam(M,\omega)$ and set $\L = p(\tL)$, where $p \colon \tSymp(M,\omega) \to \Symp_0(M,\omega)$ is the universal covering map.
Consider the following commutative diagram:
\[
\xymatrix{
  \tHam(M,\omega) \ar[r]^-{\tio} \ar[d]^-{p} & \tL \ar[r]^-{\til} \ar[d]^-{p} & \tSymp(M,\omega) \ar[d]^-{p} \\
  \Ham(M,\omega) \ar[r]^-{i_0} & \L \ar[r]^-{i_1} & \Symp_0(M,\omega).
}
\]
Here $\tio, \til, i_0$ and $i_1$ are the inclusions.

\begin{lem} \label{lem:subgroup_ext_zero}
  The following are equivalent.
  \begin{enumerate}[$(1)$]
    \item $i_1^* R_0 \in \HHH^2(\L)$ is zero.
    \item There exists $\phi \in \QQQ(\L)$ such that $\tio^*p^*\phi = \sqm_M$.
  \end{enumerate}
\end{lem}

\begin{proof}
  To prove that (1) implies (2), we assume that $i_1^*R_0 = 0$.
  Recall that $b_J$ is a cocycle representing the Reznikov class.
  Then there exists $u \in \CCC^1(\L)$ such that $i_1^*b_J = \delta u$.
  Since $b_J$ is a bounded cocycle, $u$ is a quasimorphism.
  Recall from \eqref{bounded_coboundary} that $-\delta \nu_J = p^* i_0^* i_1^*b_J$.
  Hence we have
  \[
    -\delta \nu_J = p^* i_0^* \delta u = \delta (p^* i_0^* u) = \delta (\tio^* p^* u).
  \]
  This implies that $\nu_J + \tio^* p^* u \colon \tHam(M,\omega) \to \RR$ is a homomorphism.
  Because $\tHam(M,\omega)$ is perfect (\cite{Ban}), we have $\nu_J = - \tio^* p^* u$.
  Let $\phi$ be the homogenization of $-u$.
  Then we have $\sqm_M = \tio^* p^* \phi$.

  To prove that (2) implies (1), we assume (2) and take $\phi \in \QQQ(\L)$ satisfying $\tio^*p^*\phi = \sqm_M$.
  Since $\sqm_M$ is the homogenization of $\nu_J$, there exists a bounded function $v \colon \tHam(M,\omega) \to \RR$ such that $\sqm_M = \nu_J + v$.
  Then we have 
  \[
    -\delta \nu_J = -\delta(\sqm_M - v) = \delta v - \delta \tio^* p^* \phi = \delta v - p^* i_0^* \delta \phi.
  \]
  Together with $-\delta \nu_J = p^* i_0^* i_1^*b_J$, we have 
  \[
    p^*i_0^*(i_1^*b_J + \delta \phi) = \delta v.
  \]
  Note that $i_1^*b_J + \delta \phi$ is a bounded cocycle on $\L$.
  Since $v$ is a bounded function, the second bounded cohomology class $p^*i_0^*[i_1^*b_J + \delta \phi]$ of $\tHam(M,\omega)$ is zero.
  Since $\L/\Ham(M,\omega)$ is abelian, the bounded cohomology $\HHH_b^2(\L/\Ham(M,\omega))$ is zero. 
  In particular, the map $i_0^* \colon \HHH_b^2(\L) \to \HHH_b^2(\Ham(M,\omega))$ is injective by \eqref{5term_bdd}.
  Moreover, since $p \colon \tHam(M,\omega) \to \Ham(M,\omega)$ is surjective, the map $p^* \colon \HHH_b^2(\Ham(M,\omega)) \to \HHH_b^2(\tHam(M,\omega))$ is also injective by \eqref{5term_bdd}.
  Hence, the triviality of $p^*i_0^*[i_1^*b_J + \delta \phi]$ implies that $[i_1^*b_J + \delta \phi] = 0$ in $\HHH_b^2(\L)$.
  In particular, $[i_1^*b_J] = i_1^* R_b$ coincides with $-[\delta \phi]$.
  Let $c_{\L} \colon \HHH_b^2(\L) \to \HHH^2(\L)$ be the comparison map.
  Since $c_{\L}([\delta \phi]) = 0$, we have $i_1^* R = 0$.
\end{proof}

\begin{proof}[Proof of Proposition \textup{\ref{prop:non-zero_R}}]
  Let us consider the case where $\tL = \tSymp(M,\omega)$.
  Assume that $\mathfrak{S}_M$ on $\tHam(M, \omega)$ is not extendable to $\tSymp(M,\omega)$.
  Then, there does not exist $\phi \in \QQQ(\Symp_0(M,\omega))$ such that $\tio^*p^*\phi = \sqm_M$.
  Hence, by Lemma \ref{lem:subgroup_ext_zero}, the Reznikov class $R_0$ is non-zero.
\end{proof}

\begin{proof}[Proof of Theorem \textup{\ref{reznikov_main}}]
  Theorem \ref{main_non_ext} and Proposition \ref{prop:non-zero_R} imply the theorem.
\end{proof}

Applying $\tL = \tHam(M,\omega)$ to Lemma \ref{lem:subgroup_ext_zero}, we obtain the following.
\begin{cor}
  Let $(M, \omega)$ be a closed symplectic manifold.
  Then $\mathfrak{S}_M$ is non-descendible to $\Ham(M,\omega)$ if and only if the Reznikov class $R$ is non-zero on $\Ham(M,\omega)$.
\end{cor}

Combining the above corollary and Lemma \ref{lem:subgroup_ext_zero} for the case of $\tL = \tSymp(M,\omega)$, we have the following.
\begin{cor}\label{equiv_RS}
  Let $(M,\omega)$ be a closed symplectic manifold.
  The following are equivalent.
  \begin{enumerate}[$(1)$]
 \item Shelukhin's quasimorphism $\mathfrak{S}_M$ is descendible to $\Ham(M,\omega)$ and the descended quasimorphism is not extendable to $\Symp_0(M,\omega)$.
  \item The Reznikov class is equal to zero on $\Ham(M,\omega)$ and non-zero on $\Symp_0(M,\omega)$
\end{enumerate}
\end{cor}

\begin{remark}
  If Shelukhin's quasimorphism $\mathfrak{S}_M \colon \tHam(M,\omega) \to \RR$ is non-descendible, then the Reznikov class restricted to $\Ham(M,\omega)$ is a non-zero image of the map
  \[
    \HHH^2(B\Ham(M,\omega)) \to \HHH^2(\Ham(M,\omega))
  \]
  by  Corollary 1.2 in \cite{KM21}.
  Here the domain of the map is the singular cohomology of the classifying space $B\Ham(M,\omega)$ and the codomain of the map is the group cohomology of $\Ham(M,\omega)$, which is canonically isomorphic to the singular cohomology of the classifying space $B\Ham(M,\omega)^{\delta}$ of the group $\Ham(M,\omega)$ equipped with discrete topology.
  Hence, in this case, the Reznikov class gives rise to a characteristic class of (possibly non-foliatable) Hamiltonian fibration.
\end{remark}

We have the following corollary to Corollary \ref{equiv_RS}.
\begin{cor}\label{cor:surface_Reznikov}
Let $(S,\omega_S)$ be a closed 
orientable surface whose genus is at least two with a symplectic form $\omega$.
Then, the Reznikov class is equal to zero on $\Ham(S,\omega)$ and non-zero on $\Symp_0(S,\omega)$.
\end{cor}
\begin{proof}
It is known that $\Ham(S,\omega)$ is simply connected (for example, see \cite[Subsection 7.2]{P01})
and hence Shelukhin's quasimorphism $\mathfrak{S}_S$ is descendible to $\Ham(S,\omega)$.
By Theorem \ref{main_non_ext},  the descended quasimorphism is not extendable to $\Symp_0(S,\omega)$.
Thus, by Corollary \ref{equiv_RS}, we complete the proof.
\end{proof}

Recall from Subsection \ref{subsec:sh_qm} that the restriction of $\sqm_M$ to $\pi_1(\Ham(M,\omega))$ coincides with the homomorphism $I_{c_1} \colon \pi_1(\Ham(M,\omega)) \to \RR$.
Hence the non-descendibility of Shelukhin's quasimorphism is equivalent to the non-triviality of the homomorphism $I_{c_1}$.
Hence we obtain the following.
  
\begin{cor}
  If the homomorphism $I_{c_1} \colon \pi_1(\Ham(M,\omega)) \to \RR$ is non-zero, then the Reznikov class is also non-zero.
\end{cor}

\subsection{Extension of the Reznikov class to the group of volume-preserving diffeomorphisms}

If the dimension of the symplectic manifold $(M,\omega)$ is equal to $2n$, then $\Omega = \omega^n$ is a volume form on $M$.
In this subsection, we show that, in some cases (e.g., compact K\"{a}hler manifolds), the Reznikov class restricted to $\Symp_0(M,\omega)$ (resp. its pullback to the universal covering group $\tSymp(M,\omega)$) extends to the identity component $\Diff_0(M,\Omega)$ (resp. the universal covering group $\tDiff(M,\Omega)$ of the group of volume-preserving diffeomorphisms).
Moreover, we show that the boundedness of the Reznikov class and its extension to the group of volume-preserving diffeomorphisms is sensitive to the difference between $\tDiff(M,\Omega)$ and $\tDiff(M,\Omega)$.

Let $\tR \in \HHH^2(\tSymp(M,\omega))$ and $\tRb \in \HHH_b^2(\tSymp(M,\omega))$ denote the pullbacks of the Reznikov class and the bounded Reznikov class, respectively.
The following property of the Reznikov class is essential in this subsection.
\begin{lem}\label{lem:R_in_image_flux}
  Let $(M,\omega)$ be a closed symplectic manifold.
  Then $\tR$ is contained in
  \[
    \mathrm{Im} (c_{\tSymp(M,\omega)}) \cap \mathrm{Im} (\widetilde{\flux}_{\omega}^*).
  \]
  Here $c_{\tSymp(M,\omega)} \colon \HHH_b^2(\tSymp(M,\omega)) \to \HHH^2(\tSymp(M,\omega))$ is the comparison map and $\widetilde{\flux}_{\omega}^* \colon \HHH^2(\HHH_c^1(M)) \to \HHH^2(\tSymp(M,\omega))$ is the pullback by the flux homomorphism.
\end{lem}
\begin{proof}

The condition $\tR \in \mathrm{Im} (c_{\tSymp(M,\omega)})$ immediately follows from $\tR =c_{\tSymp(M,\omega)}(\tRb)$.
Hence, it is sufficient to prove $\tR \in \mathrm{Im} (\widetilde{\flux}_{\omega}^*)$.

Let us consider the exact sequence
\begin{align}\label{ex1}
  1 \to \tHam(M,\omega) \xrightarrow{\widetilde{i_0}} \tSymp(M,\omega) \xrightarrow{\widetilde{\flux}_{\omega}} \HHH_c^1(M) \to 0.
\end{align}
Since $\tHam(M,\omega)$ is perfect, the first cohomology $\HHH^1(\tHam(M,\omega))$ is trivial.
Hence, the exact sequence \eqref{5term} applied to \eqref{ex1} implies that the sequence
\[
  \HHH^2(\HHH_c^1(M)) \xrightarrow{\widetilde{\flux}^*} \HHH^2(\tSymp(M,\omega)) \xrightarrow{\tio^*} \HHH^2(\tHam(M,\omega))^{\tSymp(M,\omega)}
\]
is exact.
By \eqref{bounded_coboundary} and the definition of the Reznikov class, we have $\tio^* \tR = 0$, which implies that $\tR \in \mathrm{Im} (\widetilde{\flux}_{\omega}^*)$.
\end{proof}

Let us consider the following condition:
\begin{align}\label{lefschetz}
  \smile [\omega]^{n-1} \colon \HHH^1(M) \to \HHH^{2n-1}(M) \text{ is an isomorphim}.
\end{align}
This condition is a part of the Lefschetz property, and hence compact K\"{a}hler manifolds satisfy \eqref{lefschetz}.

\begin{prop}\label{prop:R_extend}
  If a closed symplectic manifold $(M,\omega)$ satisfies \eqref{lefschetz}, then $\tR \in \HHH^2(\tSymp(M,\omega))$ extends to a class $\tR^{\vol}$ in $\HHH^2(\tDiff(M,\Omega))$.
\end{prop}

\begin{proof}
  Let us consider the commutative diagram
  \begin{align*}
    \xymatrix{
    \tSymp(M,\omega) \ar[r]^-{\widetilde{\flux}_{\omega}} \ar[d]^-{\iota} & \HHH^1(M) \ar[d]^-{\smile [\omega]^{n-1}} \\
    \tDiff(M,\Omega) \ar[r]^-{\widetilde{\flux}_{\Omega}} & \HHH^{n-1}(M),
    }
  \end{align*}
  where $\widetilde{\flux}_{\Omega}$ is the volume flux homomorphism (see \cite[Chapter 3]{Ban97} for the definition).
  By Lemma \ref{lem:R_in_image_flux}, we take a class $u \in \HHH^2(\HHH^1(M))$ satisfying $\widetilde{\flux}_{\omega}^*(u) = \tR$.
  We set
  \[
    \tR^{\vol} = \widetilde{\flux}_{\Omega}^* \circ ((\smile [\omega]^{n-1})^*)^{-1} (u).
  \]
  By the commutativity of the diagram, we have $\iota^*(\tR^{\vol}) = \tR$.
\end{proof}

Recall that $\HHH_b^2(\HHH^1(M)) = 0$ since $\HHH^1(M)$ is abelian.
Hence the argument in the proof of Proposition \ref{prop:R_extend} cannot be applied to the bounded cohomology class $\tRb$.
In fact, under an additional assumption on Shelukhin's quasimorphism, the following holds.

\begin{prop} \label{prop_vol_reznikov_unbdd}
  Let $(M,\omega)$ be a closed symplectic manifold with \eqref{lefschetz} and $\dim(M) \geq 4$.
  If Shelukhin's quasimorphism $\mathfrak{S}_{M} \colon \tHam(M,\omega) \to \RR$ is not extendable to $\tSymp(M,\omega)$, then the class $\tR^{\vol}$ defined in the proof of Proposition \textup{\ref{prop:R_extend}} does not belong to the image of the comparison map
  \[
    c_{\tDiff(M,\Omega)} \colon \HHH_b^2(\tDiff(M,\Omega)) \to \HHH^2(\tDiff(M,\Omega)).
  \]
\end{prop}

\begin{proof}
  The argument similar to Lemma \ref{lem:subgroup_ext_zero} shows that the non-extendability of $\mathfrak{S}_M$ induces the non-triviality of the class $\tR$.
  Since $\iota^*(\tR^{\vol}) = \tR$, the class $\tR^{\vol}$ is also non-zero.
  By construction, the class $\tR^{\vol}$ is contained in $\mathrm{Im}(\widetilde{\flux}_{\Omega}^*)$.
  Due to the fact that the volume flux homomorphism $\widetilde{\flux}_{\Omega}$ has a section homomorphism
  \[
    s \colon \HHH_c^1(M) \to \tDiff(M,\Omega)
  \]
   if $\dim(M) \geq 3$ (\cite[Proposition 6.1]{Fa}), by \eqref{lefschetz}, the class
\[
s^*(\tR^{\vol})
 = s^* \circ \widetilde{\flux}_{\Omega}^* \circ ((\smile [\omega]^{n-1})^*)^{-1} (u) =((\smile [\omega]^{n-1})^*)^{-1} (u)
  \in \HHH^2(\HHH_c^1(M))
\]
 is non-zero.
  The commutative diagram
  \begin{align*}
    \xymatrix{
    \HHH_b^2(\tDiff(M,\Omega)) \ar[r]^-{s^*} \ar[d]^-{c_{\tDiff(M,\Omega)}} & \HHH_b^2(\HHH_c^1(M))= 0 \ar[d] \\
    \HHH^2(\tDiff(M,\Omega)) \ar[r]^-{s^*} & \HHH^2(\HHH_c^1(M)),
    }
  \end{align*}
  verifies that the class $\tR^{\vol}$ is not in the image of $c_{\tDiff(M,\Omega)}$.
\end{proof}

\begin{cor}\label{non_inj}
   Under the assumptions of Proposition~\textup{\ref{prop_vol_reznikov_unbdd}},
  the induced map $\iota^* \colon \HHH_{/b}^2(\tDiff(M,\Omega)) \to \HHH_{/b}^2(\tSymp(M,\omega))$ is not injective.
\end{cor}

\begin{proof}

Let $q\colon \HHH^2(\tDiff(M,\Omega)) \to \HHH_{/b}^2(\tDiff(M,\Omega))$ be the map induced by the quotient map $C^\ast\left(\tDiff(M,\Omega)\right) \to C^\ast\left(\tDiff(M,\Omega)\right)/C_b^\ast\left(\tDiff(M,\Omega)\right)$.
  Then, we have $q(\tR^{\vol}) \neq 0$ by Proposition \ref{prop_vol_reznikov_unbdd}.
  Since $\iota^*(\tR^{\vol}) = \tR \in \HHH^2(\tSymp(M,\omega))$ is a bounded class, we have $\iota^*(q(\tR^{\vol})) = q(\iota^*(\tR^{\vol})) = 0$.
\end{proof}

\begin{proof}[Proof of Theorem \textup{\ref{main_not_inj}}]
Under the assumption of Theorem \textup{\ref{main_not_inj}}, $(M,\omega)$ is a K\"{a}hler manifold and thus has the property \eqref{lefschetz}.
By Theorem \ref{main_non_ext}, Shelukhin's quasimorphism $\mathfrak{S}_{M} \colon \tHam(M,\omega) \to \RR$ is not extendable to $\tSymp(M,\omega)$.
Hence, Corollary \ref{non_inj} implies that $\iota^* \colon \HHH_{/b}^2(\tDiff(M,\Omega)) \to \HHH_{/b}^2(\tSymp(M,\omega))$ is not injective.
\end{proof}

\section*{Acknowledgment}

The first-named author, the second-named author, the fourth-named author and the fifth-named author are partially supported by JSPS KAKENHI Grant Number JP21K13790, JP24K16921, JP23K12975 and JP21K03241, respectively.
The third-named author is partially supported by JSPS KAKENHI Grant Number JP23KJ1938 and JP23K12971.




\bibliographystyle{amsalpha}
\bibliography{KKMM0102}

\end{document}